\documentclass[reqno,11pt]{amsart}
\usepackage{amsmath,epsfig}
\usepackage{amssymb}

\newtheoremstyle{rem}{1.3ex}{1.3ex}{\rmfamily}{}
{\itshape\rmfamily}{}{1.5ex}{}

\newtheorem{theorem}{Theorem}[section]
\newtheorem{lemma}[theorem]{Lemma}

\newtheorem{cor}[theorem] {Corollary}

\theoremstyle{definition}
\newtheorem{example}[theorem] {Example}
\newtheorem{remark}[theorem] {Remark}

\renewcommand{\section}{\secdef\sct\sect}
\newcommand{\sct}[2][default]{\refstepcounter{section}
\setcounter{equation}{0}
\vspace{0.5cm}
\centerline{ \large
\scshape \arabic{section}. #1}
\vspace{0.3cm}}
\newcommand{\sect}[1]{
\vspace{0.5cm}
\centerline{\large\scshape #1}
\vspace{0.3cm}}

\renewcommand{\subsection}{\secdef \subsct\sbsect}
\newcommand{\subsct}[2][default]{\refstepcounter{subsection}
\nopagebreak
\vspace{0.5\baselineskip}
{\flushleft\bf \arabic{section}.\arabic{subsection}~\bf #1  }
\nopagebreak}
\newcommand{\sbsect}[1]{\vspace{0.1cm}\noindent
{\bf #1}\vspace{0.1cm}}

\newcommand{\Hcal}{{\mathcal {H}}}

\newcommand{\const}{{\operatorname {const.}\,}}
\renewcommand{\b}     {\beta}

\newcommand{\R}     {\mathbb{R}}
\newcommand{\Z}     {\mathbb{Z}}
\newcommand{\N}     {\mathbb{N}}

\newcommand{\E}     {\mathbb{E}}

\def\1{{\mathchoice {1\mskip-4mu\mathrm l}
                    {1\mskip-4mu\mathrm l}
                    {1\mskip-4.5mu\mathrm l} {1\mskip-5mu\mathrm l}}}

\begin{document}

%\today
\title[Stein's method for dependent random variables occurring in Statistical Mechanics]{\large
Stein's method for dependent random variables \\occurring in Statistical Mechanics}

\author[Peter Eichelsbacher and Matthias L\"owe]{}
\maketitle
\thispagestyle{empty}
\vspace{0.2cm}

\centerline{\sc Peter Eichelsbacher\footnote{Ruhr-Universit\"at Bochum, Fakult\"at f\"ur Mathematik,
NA3/68, D-44780 Bochum, Germany, {\tt Peich@math.ruhr-uni-bochum.de  }}
and Matthias L\"owe\footnote{Universit\"at M\"unster, Fachbereich Mathematik und Informatik, Institut f\"ur Mathematische Statistik,
Einsteinstr. 62, D-48149 M\"unster, Germany {\tt maloewe@math.uni-muenster.de}}}
\vspace{2 cm}

%S\centerline{\sc DO NOT DISTRIBUTE THE PREPRINT, PLEASE}

%\centerline{\small{\version}}
\vspace{1cm}

\begin{quote}
{\small {\bf Abstract:} We obtain rates of convergence in limit theorems of partial
sums $S_n$ for certain sequences of dependent, identically distributed random variables, which arise
naturally in statistical mechanics, in particular, in the context of the Curie-Weiss models. Under appropriate assumptions
there exists a real number $\alpha$, a positive real number $\mu$, and a positive integer $k$ such that
$(S_n- n \alpha)/n^{1 - 1/2k}$ converges weakly to a random variable with density proportional to $\exp(  -\mu |x|^{2k} /(2k)! )$.
We develop Stein's method for exchangeable pairs for a rich class of distributional
approximations including the Gaussian distributions as well as the  non-Gaussian limit distributions with density
proportional to  $\exp(  -\mu |x|^{2k} /(2k)! )$. Our results include the optimal Berry-Esseen rate in the
Central Limit Theorem for the total magnetization in the classical Curie-Weiss model, for high temperatures as well as at the critical
temperature $\beta_c=1$, where the Central Limit Theorem fails. Moreover, we analyze Berry-Esseen bounds as the temperature $1/ \beta_n$
converges to one and obtain a threshold for the speed of this convergence. Single spin distributions satisfying
the Griffiths-Hurst-Sherman (GHS) inequality like models of liquid helium or continuous Curie-Weiss models are
considered. %Correlation inequalities for even ferromagnets will be applied. DER SATZ IST IM FALSCHEN TEMPUS; ABER ICH WEISS AUCH GAR NICHT WAS ER HEISST ...
%WAS IST EIN EVEN FERROMAGNET? ... ICH HABE ES MAL GESTRICHEN.
}
\end{quote}

\vfill



\medskip\noindent
This research was  done partly at the Mathematisches Forschungsinstitut Oberwolfach during a stay within the Research in Pairs
Programme from August 24 - September 6, 2008.
\eject

\setcounter{section}{0}

\section{Introduction and main result}\label{intro}

There is a long tradition in considering mean--field models in statistical
mechanics. The Curie--Weiss model is famous, since it exhibits a number of
properties of real substances, such as multiple phases, metastable states and others,
explicitly. The aim of this paper
is to prove Berry-Esseen bounds for the sums of dependent random variables
occurring in statistical mechanics under the name Curie-Weiss models. To this end, we
will develop Stein's method for exchangeable pairs  (see \cite{Stein:1986})
for a rich class of distributional approximations.
For an overview
of results on the Curie--Weiss models and related models, see \cite{Ellis:LargeDeviations},
\cite{Ellis/Newman:1978}, \cite{Ellis/Newman/Rosen:1980}.

For a fixed positive integer $d$ and a finite subset $\Lambda$ of $\Z^d$, a ferromagnetic
crystal is described by random variables $X_i^{\Lambda}$ which represent the spins of the atom
at sites $i \in \Lambda$, where $\Lambda$ describes the macroscopic shape of the crystal.
In {\it Curie--Weiss} models, the joint distribution at fixed temperature $T >0$
of the spin random variables is given by

\begin{equation} \label{CWmeasure}
P_{\Lambda, \beta}((x_i)):=P_{\Lambda, \beta} \bigl( (X_i^{\Lambda})_{i \in \Lambda} = (x_i)_{i \in \Lambda} \bigr)
:=  \frac{1}{Z_{\Lambda}(\beta)} \exp
\biggl( \frac{\beta}{2 |\Lambda|} \bigl( \sum_{i \in \Lambda} x_i \bigr)^2
\biggr) \prod_{i \in \Lambda} \, d \varrho(x_i).
\end{equation}

Here $\beta := T^{-1}$ is the inverse temperature and $Z_{\Lambda}(\beta)$ is a normalizing
constant known as the partition function and $|\Lambda|$ denotes the cardinality of $\Lambda$.
Moreover $\varrho$ is the distribution of a single spin in the limit $\beta \to 0$.
We define $S_{\Lambda}=\sum_{i \in \Lambda} X_i^{\Lambda}$, the {\it total magnetization}
inside $\Lambda$. We take without loss of generality $d=1$ and $\Lambda = \{1, \ldots, n \}$,
where $n$ is a positive integer. We write $n$, $X_i^{(n)}$, $P_{n,\beta}$ and $S_n$, respectively, instead of $|\Lambda|$, $X_i^{\Lambda}$, $P_{\Lambda,\beta}$, and
$S_\Lambda$, respectively.
In the case where $\beta$ is fixed we may even sometimes simply write $P_n$.

We assume that $\varrho$ is in the class $\mathcal{B}$ of non-degenerate symmetric Borel probability measures on
$\R$
which satisfy
\begin{equation} \label{condrho}
\int \exp \biggl( \frac{b \, x^2}{2} \biggr) \, d \varrho (x) < \infty \quad \text{for all} \quad b>0.
\end{equation}

In the classical Curie--Weiss model, spins are distributed in  $\{-1, +1\}$
according to $\varrho = \frac 12 (\delta_{-1} + \delta_1)$.
More generally, the Curie--Weiss model carries an
additional parameter $h >0$ called {\it external magnetic field} which leads
to the modified measure, given by
$$
P_{n,\beta,h} (x) = \frac{1}{Z_{n, \beta,h}} \exp \bigl( \frac{\beta}{2n} S_n^2 +
\beta \, h S_n  \bigr) \, d \varrho^{\otimes n} (x), \quad x=(x_i).
$$
The measures $P_{n, \beta, h}$ is completely determined by the value of the total magnetization. It is therefore called an {\it order parameter}
and its behaviour will be studied in this paper.
The non-negative external magnetic field strength may even depend on the site:
\begin{equation} \label{zn}
P_{n, \beta, h_1, \ldots, h_n}(x) = \frac{1}{Z_{n , \beta, h_1, \ldots, h_n}}
\exp \bigl( \frac{\beta}{2n} S_n^2 +
\beta \, \sum_{i=1}^n h_i \, x_i  \bigr) \, d \varrho^{\otimes n} (x), \quad x=(x_i).
\end{equation}

In the general case \eqref{CWmeasure}, we will see (analogously to the treatment in
\cite{Ellis/Newman:1978, Ellis/Newman/Rosen:1980}) that the
asymptotic behaviour of $S_n$ depends crucially on the extremal
points of a function $G$ (which is a transform of the rate function in a corresponding
large deviation principle): define
$$
\phi_{\varrho}(s) := \log \int \exp (s \, x) \, d \varrho(x)
$$
and
\begin{equation} \label{Gdef}
G_{\varrho}(\beta,s) := \frac{\beta \, s^2}{2} - \phi_{\varrho}(\beta \, s).
\end{equation}
%The connection between the free energy $f_{\varrho}(\cdot)$ and $G_{\varrho}(\cdot)$
%is
%$$f_{\varrho}(\beta) = \frac{1}{\beta} \inf \{ G_{\varrho}(\beta,s);
%s \operatorname{real} \}.
%$$
We shall drop $\beta$ in the notation for $G$ whenever there is no danger
of confusion, similarly we will suppress $\varrho$ in the notation for $\phi$ and $G$.
For any measure $\varrho \in \mathcal{B}$, $G$ was proved to have global minima, which can be
only finite in number, see \cite[Lemma 3.1]{Ellis/Newman:1978}. Define $C = C_{\varrho}$ to be
the discrete, non--empty set of
%real numbers $m$ such that $m$ is either a global minimum, a local
minima (local or global)
%or a point of inflection
of $G$. If $\alpha \in C$,
then there exists a positive integer $k:=k(\alpha)$ and a positive real number $\mu:= \mu(\alpha)$
such that
\begin{equation}\label{Taylor}
G(s) = G(\alpha) + \frac{\mu(\alpha) (s-\alpha)^{2k}}{(2k)!} + {\mathcal O}((s-\alpha)^{2k+1} ) \quad \text{as}
\quad s \to \alpha.
\end{equation}
%If $m \in C$ is a point of inflection of $G$, a similar expansion holds. In either case,
The numbers $k$ and $\mu$ are called the {\it type} and {\it strength}, respectively, of the extremal
point $\alpha$. Moreover, we define the maximal type $k^*$ of $G$ by the formula
$$
k^* = \max \{ k(\alpha); \alpha \text{ is a global minimum of } G \}.
$$
Note that the $\mu(\alpha)$ can be calculated explicitly: one gets
\begin{equation}\label{k=1}
\mu(\alpha) = \beta- \b^2\phi''(\beta \, \alpha)\qquad \mbox{if  }k=1
\end{equation}
while
\begin{equation}\label{kriesig}
\mu(\alpha) = - \beta^{2k} \phi^{(2k)}(\beta \, \alpha) \qquad \mbox{if  }k\ge 2
\end{equation}
(see \cite{Ellis/Newman/Rosen:1980}).
%Sometimes it will be necessary to consider further terms in the
%expansion of $G$ we will then write
%\begin{equation}\label{Taylor2}
%G(s) = G(m) + \sum_{k=1}^{K} \frac{\mu_k(m) (s-m)^{2k}}{(2k)!} + {\mathcal O}((s-m)^{2K+1} ) \quad \text{as}
%\quad s \to m
%\end{equation}
%and $\l_k=\l_k(m)$.

An interesting point is, that the global minima
of $G$ of maximal type correspond to stable states, meaning that multiple minima
represent a mixed phase and a unique global minimum a pure phase. For details see
the discussions in  \cite{Ellis/Newman/Rosen:1980}.

The following is known about the fluctuation behaviour of $S_n$ under $P_n$.
In the classical model ($\varrho$ is the symmetric Bernoulli measure), for
$0 < \beta < 1$, in \cite{Ellis/Newman:1978} the
Central Limit Theorem is proved:
$$
\frac{\sum_{i=1}^n X_i}{\sqrt{n}} \to N(0, \sigma^2(\beta))
$$
in distribution with respect to the Curie--Weiss finite volume Gibbs states
with $\sigma^2(\beta) = (1-\beta)^{-1}$.
Since for $\beta = 1$ the variance $\sigma^2(\beta)$ diverges, the
Central Limit Theorem fails at the critical point. In \cite{Ellis/Newman:1978} it is proved that
for $\beta = 1$ there exists a random variable $X$ with probability density proportional to $\exp(- \frac{1}{12} x^4)$
such that as $n \to \infty$
$$
\frac{\sum_{i=1}^n X_i}{n^{3/4}} \to X
$$
in distribution with respect to the finite-volume Gibbs states. Asymptotic
independence properties and propagation of chaos for blocks of size $o(n)$ have been
investigated in \cite{BenArous/Zeitouni:1999}.

In general, given $\varrho \in \mathcal{B}$, let $\alpha$ be one of the global minima of maximal type $k$ and strength
$\mu$ of $G_{\varrho}$. Then
$$
\frac{S_n - n \alpha}{n^{1 - 1/2k}} \to X_{k, \mu, \beta}
$$
in distribution, where $X_{k, \mu, \beta}$ is a random variable with probability density $f_{k, \mu, \beta}$, defined by
\begin{equation} \label{densitysigma}
f_{1, \mu, \beta}(x) = \frac 1 {\sqrt{2 \pi \sigma^2} }\exp \bigl( -x^2 / 2 \sigma^2 \bigr)
\end{equation}
and for $k \geq 2$
\begin{equation} \label{densitygen}
f_{k, \mu, \beta}(x) = \frac{\exp \bigl( - \mu x^{2k} / (2 k)!  \bigr)}{ \int \exp \bigl( - \mu x^{2k} / (2k)! \bigr) \, dx}.
\end{equation}
Here, $\sigma^2 = \frac 1 \mu - \frac 1\beta$ so that for $\mu = \mu(\alpha)$ as in \eqref{k=1}, $\sigma^2 = ([ \phi''(\beta \alpha)]^{-1} - \beta)^{-1}$
(see \cite{Ellis/Newman:1978}, \cite{Ellis/Newman/Rosen:1980}). Moderate deviation principles have been investigated in \cite{Eichelsbacher/Loewe:2004}.
\medskip

In \cite{Ellis/Monroe/Newman:1976} and \cite{Ellis/Newman/Rosen:1980}, a class of measures $\varrho$ is described exhibiting a behaviour
similar to that of the classical Curie--Weiss model. Assume that $\varrho$ is any symmetric measure that satisfies the
Griffiths-Hurst-Sherman (GHS) inequality,
\begin{equation}\label{GHS}
\frac{d^3}{ds^3} \phi_\varrho(s) \le 0 \quad \mbox{for all } s\ge 0,
\end{equation}
(see also \cite{Ellis/Newman:1978b, Griffiths:1970}). One can show that in this case
$G$ has the following properties:
There exists a value $\beta_c$, the inverse critical temperature,
% and this inverse critical temperature equals one.
and $G$ has a unique
global minimum at the origin for $0 < \beta \leq \beta_c$ and exactly two global minima, of equal type,
for $\beta >\beta_c$. For $\beta_c$ the unique global minimum is of type $k \geq 2$ whereas
for $\beta \in (0,\beta_c)$ the unique global minimum is of type 1. At $\beta_c$ the law
of large numbers still holds, but the fluctuations of $S_n$ live on a smaller scale than
$\sqrt n$. This critical temperature can be explicitly computed as $\beta_c= 1 / \phi''(0) = 1 / \operatorname{Var}_{\varrho}(X_1)$.
By rescaling the $X_i$ we may thus assume that $\beta_c=1$.

Alternatively, the GHS-inequality can be formulated in the terms of $Z_{n,\beta, h_1, \ldots, h_n}$, defined in \eqref{zn}:
\begin{eqnarray} \label{GHS2}
0 & \geq & \frac{\partial^3}{\partial h_i \, \partial h_j \partial h_k} \log Z_{n, \beta, h_1, \ldots, h_n} \nonumber \\
& = & E(X_i X_j X_k) - \E(X_i) \E(X_j X_k) - \E(X_j) \E(X_i X_k) \\ &  & \hspace{0.5cm} - \E(X_k) \E(X_i X_j) + 2 \E(X_i) \E(X_j) \E(X_k) \nonumber
\end{eqnarray}
for all (not necessarily distinct) sites $i,j,k \in \{1, \ldots, n\}$. Here $\E$ denotes the expectation with respect
to $P_{n, \beta, h_1, \ldots, h_n}$. The GHS inequality has a number of interesting implications, see
\cite{Ellis/Monroe/Newman:1976}.

With GHS, we will denote the set of measures $\varrho \in \mathcal{B}$ such that the GHS-inequality \eqref{GHS}
is valid (for $P_{n, \beta, h_1, \ldots, h_n}$ in the sense of \eqref{GHS2}). We will give examples in Section 7.
% the statement of Theorem \ref{CWgendep}.
\medskip

\begin{remark}
In \cite[Lemma 4.1]{Ellis/Newman:1978}, for $\varrho \in \mathcal{B}$ it is proved that
$G$ has a unique global minimum if and only if
$$
\int \exp (s \, x) \, d\varrho(x) < \exp(s^2/2), \quad \text{for} \,\, s \,\, \text{real},
$$
where the right hand side of this strict inequality is the moment generating function of a standard
normal random variable. Moreover, in the same Lemma it is proved that $G$ has a local minimum at the origin
of type $k$ and strength $\mu$ if and only if
$$
\bar{\mu}_j - \mu_j(\varrho) = \left\{ \begin{array}{ll} 0 & \mbox{for } j=0,1,\ldots,2k-1, \\
\mu>0 & \mbox{for } j=2k. \\ \end{array} \right.
$$
Here $\mu_j(\varrho)$ and $\bar{\mu}_j$ define the $j$'th moment of $\varrho$ and the $j$'th
moment of a standard normal random variables, respectively.
Note that this in particular implies $\mu_1(\varrho)=\mathbb{E}_\varrho(X_1)=0$.
\end{remark}
\medskip

The aim of this paper is to prove the following theorems:

\subsection{Results for the classical Curie-Weiss model}

\begin{theorem}[classical Curie-Weiss model, optimal Berry-Esseen bounds outside the critical temperature]\label{CW}
Let $\varrho = \frac 12 \delta_{-1} + \frac 12 \delta_1$ and $0 < \beta <1$. We have
\begin{equation} \label{Kol}
\sup_{z \in \R} \bigg| P_n \biggl( S_n/ \sqrt{n} \leq z \biggr) - \Phi_{\beta}(z) \bigg| \leq C \, n^{-1/2},
\end{equation}
where $\Phi_{\beta}$ denotes the distribution function of the normal distribution with expectation zero and variance $(1- \beta)^{-1}$, and
$C$ is an absolute constant, depending on $\beta$, only.
\end{theorem}

\begin{theorem}[classical Curie-Weiss model, optimal Berry-Esseen bounds at the critical temperature]\label{CWcritical}
Let $\varrho = \frac 12 \delta_{-1} + \frac 12 \delta_1$ and $\beta =1$. We have
\begin{equation} \label{Kol2}
\sup_{z \in \R} \bigg| P_n \biggl( S_n/ n^{3/4} \leq z \biggr) - F(z) \bigg| \leq C \, n^{-1/2},
\end{equation}
where
\begin{equation} \label{density2}
F(z) := \frac{1}{Z} \int_{- \infty}^z \exp(-x^4/12) \, dx,
\end{equation}
$Z:= \int_{\R}   \exp(-x^4/12)\, dx$ and
$C$ is an absolute constant.
\end{theorem}

\begin{theorem}[Berry-Esseen bounds for size-dependent temperatures]\label{CWbetadep}
Let $\varrho = \frac 12 \delta_{-1} + \frac 12 \delta_1$ and $0 < \beta_n < \infty$ depend on $n$ in such a way that $\beta_n \to 1$
monotonically as $n \to \infty$. Then the following assertions hold:
\begin{enumerate}
\item
If $\beta_n -1 = \frac{\gamma}{\sqrt{n}}$ for some $\gamma \not= 0$, we have
\begin{equation} \label{Kol3}
\sup_{z \in \R} \bigg| P_n \biggl( S_n/ n^{3/4} \leq z \biggr) - F_{\gamma}(z) \bigg| \leq C\, n^{-1/2}
\end{equation}
with $$
F_{\gamma}(z) := \frac{1}{Z} \int_{- \infty}^z \exp \bigl( - \frac{x^4}{12} + \frac{\gamma x^2}{2} \bigr) \, dx.
$$
where $Z:= \int_{\R} \exp \bigl( - \frac{x^4}{12} + \frac{\gamma x^2}{2} \bigr)\, dx$ and $C$ is an absolute constant.
\item
If $|\beta_n -1| \ll n^{-1/2}$, $S_n/ n^{3/4}$ converges in distribution to $F$, given in \eqref{density2}. Moreover, if $|\beta_n -1| = \mathcal{O}(n^{-1})$,
\eqref{Kol2} holds true.
\item
If $|\beta_n -1| \gg n^{-1/2}$, the Kolmogorov distance of the distribution of $\sqrt{\frac{1-\beta_n}{n}} \sum_{i=1}^n X_i$ and
the normal distribution $N(0, (1-\beta_n)^{-1})$ converges to zero. Moreover, if $|\beta_n -1| \gg n^{-1/4}$, we obtain
\begin{equation*}
\sup_{z \in \R} \bigg| P_n \biggl( \frac{\sqrt{(1-\beta_n)} S_n}{\sqrt{n}} \leq z \biggr) - \Phi_{\beta_n}(z) \bigg| \leq C \, n^{-1/2}
\end{equation*}
with an absolute constant $C$.
\end{enumerate}
\end{theorem}

\begin{remark}
In \cite{Barbour:1980}, Barbour obtained distributional limit theorems, together with rates of convergence, for the equilibrium
distributions of a variety of one-dimensional Markov population processes. In section 3 he mentioned, that his results
can be interpreted in the framework of \cite{Ellis/Newman:1978}. As far as we understand, his result (3.9) can be interpreted
as the statement \eqref{Kol2}, but with the rate $n^{-1/4}$.
\end{remark}

\begin{remark}
In the first assertion of Theorem \ref{CWbetadep}, our method of proof allows to compare the distribution of $S_n/n^{3/4}$
alternatively with the distribution with Lebesgue-density proportional to
$$
\exp \bigl( - \frac{\beta_n^3 x^4}{12} + \frac{\gamma \, x^2}{2} \bigr).
$$
To be able to compare the distribution of interest with a distribution depending on $n$ (on $\beta_n$), is one of the
advantages of Stein's method. The proof of this statement follows immediately from the proof of Theorem \ref{CWbetadep}.

If in Theorem \ref{CWbetadep} (2) $|\beta_n-1| \gg n^{-1}$ the speed of convergence reduces to $\mathcal{O}(\sqrt n |1-\beta_n|)$.
Likewise, if in Theorem \ref{CWbetadep} (3)
$|\beta_n-1| \ll n^{-1/4}$, the speed of convergence is $\mathcal{O}(\frac 1 {n |1-\beta_n|})$. This reduced speed of
convergence reflects the influence of two potential limiting measures. Next to the ''true'' limit there is also the limit
measure from part (1) of Theorem \ref{CWbetadep}, which in these cases is relatively close to our measures of interest.
\end{remark}
\medskip

\subsection{Results for a general class of Curie-Weiss models}

More generally, we obtain Berry-Esseen bounds for sums of dependent random variables occurring in the general
Curie-Weiss models. We will be able to obtain Berry-Esseen-type results for $\varrho$-a.s. bounded single-spin
variables $X_i$:

\begin{theorem} \label{CWgeneral}
Given $\varrho \in \mathcal{B}$ in GHS, let $\alpha$ be the global minimum of type $k$ and strength $\mu$ of $G_{\varrho}$.
Assume that the single-spin random variables $X_i$ are bounded $\varrho$-a.s. In the case $k=1$ we obtain
\begin{equation} \label{auchwas}
\sup_{z \in \R} \biggl| P_n \biggl( \frac{S_n}{\sqrt{n}} \leq z \biggr) - \Phi_{W}(z) \biggr| \leq C n^{-1/2},
\end{equation}
where $W:= S_n/\sqrt{n}$ and $\Phi_{W}$ denotes the distribution function of the normal distribution with mean zero and variance $\E(W^2)$
and $C$ is an absolute constant depending on $0 < \beta < 1$.
For $k \geq 2$ we obtain
\begin{equation} \label{Kol4}
\sup_{z \in \R} \bigg| P_n \biggl( \frac{S_n - n \alpha}{n^{1 - 1/2k}}  \leq z \biggr) - \widehat{F}_{W,k}(z) \bigg| \leq C_k \, n^{-1/k}
\end{equation}
where $\widehat{F}_{W,k}(z) := \int_{- \infty}^z \widehat{f}_{W,k}(x) \, dx$ with $\widehat{f}_{W,k}$ defined
by
$$
\widehat{f}_{W,k}(x) := \frac{\exp \bigl( -\frac{x^{2k}}{2k\, \E(W^{2k})} \bigr)}{ \int \exp \bigl( -\frac{x^{2k}}{2k\, \E(W^{2k})} \bigr) \, dx}
$$
with $W := \frac{S_n - n \alpha}{n^{1- 1/2k}}$
and $C_k$ is an absolute constant.
\end{theorem}

\begin{theorem} \label{CWgendep}
Let $\varrho \in \mathcal{B}$ satisfy the GHS-inequality and assume that $\beta_c=1$. Let $\alpha$ be the global minimum
of type $k$ with $k \geq 2$ and strength $\mu_k$ of $G_{\varrho}$
and let the single-spin variable $X_i$ be bounded.
Let $0 < \beta_n < \infty$ depend on $n$ in such a way that $\beta_n \to 1$ monotonically as $n \to \infty$.
Then the following assertions hold true:
\begin{enumerate}
\item
If $\beta_n -1 = \frac{\gamma}{n^{1 - \frac 1k}}$ for some $\gamma \not= 0$, we have
\begin{equation} \label{Kol5}
\sup_{z \in \R} \bigg| P_n \biggl( \frac{S_n- n \alpha}{n^{1 - 1/2k} }\leq z \biggr) - F_{W,k, \gamma}(z) \bigg| \leq C_k\, n^{-1/k}
\end{equation}
with $$
F_{W,k, \gamma}(z) := \frac{1}{Z} \int_{- \infty}^z \exp \biggl( - c_W^{-1} \biggl( \frac{\mu_k}{(2k)!} x^{2k} - \frac{\gamma}{2} x^2 \biggr) \biggr) \, dx.
$$
where $Z:= \int_{\R} \exp \bigl( - c_W^{-1} \bigr( \frac{\mu_k}{(2k)!} x^{2k}  - \frac{\gamma}{2} x^2 \bigr) \bigr)\, dx$, with $W := \frac{S_n - n \alpha}{n^{1-1/2k}}
$,
$$
c_W := \frac{\mu_k}{(2k)!} \E(W^{2k}) - \gamma \E(W^2)
$$
and $C_k$ is an absolute constant.
\item
If $|\beta_n -1| \ll n^{-(1 - 1/k)}$, $\frac{S_n - n \alpha}{n^{1-1/2k}}$ converges in distribution to $\widehat{F}_{W,k}$, defined as in Theorem
\ref{CWgeneral}. Moreover, if $|\beta_n -1| = \mathcal{O}(n^{-1})$,
\eqref{Kol4} holds true.
\item
If $|\beta_n -1| \gg n^{-(1-1/k)}$, the Kolmogorov distance of the distribution of $W:= \sqrt{\frac{1-\beta_n}{n}} \sum_{i=1}^n X_i$ and
the normal distribution $N(0, \E(W^2))$ converges to zero. Moreover, if $|\beta_n -1| \gg n^{-(1/2-1/2k)}$, we obtain
\begin{equation*}
\sup_{z \in \R} \bigg| P_n \biggl( \frac{\sqrt{(1-\beta_n)} S_n}{\sqrt{n}} \leq z \biggr) - \Phi_{W}(z) \bigg| \leq C \, n^{-1/2}
\end{equation*}
with an absolute constant $C$.
\end{enumerate}
\end{theorem}

\begin{remark}
Since the symmetric Bernoulli law is ${\rm GHS}$, Theorems
\ref{CWgeneral} and \ref{CWgendep} include Berry-Esseen type results for this case. But these results differ from
the results in Theorem \ref{CW}, \ref{CWcritical} and \ref{CWbetadep} with respect to the limiting laws: the laws in \ref{CWgeneral}
and \ref{CWgendep} depend on moments of $W$. 
%This is the main reason why we decide to give the proofs in the symmetric Bernoulli case seperatly. 
The bounds in Theorems \ref{CW}-\ref{CWbetadep} are easier to obtain; moreover their proofs apply
Corollary \ref{corsigma} and part (2) of Theorem \ref{generaldensity} which are less involved versions of Stein's method
for exchangeable pairs.
\end{remark}

For arbitrary $\varrho \in {\rm GHS}$ we are able to proof good bounds with respect to the Wasserstein-metric. 
For any class of test functions $\mathcal{H}$, a distance on probability measures on $\R$ can be defined by
$$
d_{\mathcal{H}} (P,Q) = \sup_{h \in {\mathcal H}} \bigg| \int h \, dP - \int h \, dQ \bigg|.
$$
The class of test functions $h$ for the Wasserstein distance $d_{w}$ is just the Lipschitz functions ${\rm Lip}(1)$ with constant no greater
than 1. The total variation distance is given by the set ${\mathcal H}$ of indicators of Borel sets, the Kolmogorov
distance $d_{K}$ by the set of indicators of half lines.

Only for technical reasons, we consider now a modified model. Let
$$
\widehat{P}_{n, \beta, h}(x) = \frac{1}{\widehat{Z}_{n, \beta, h}} \exp \biggl( \frac{\beta}{n} \sum_{1 \leq i < j \leq n} x_i x_j + \beta \, h \sum_{i=1}^n x_i \biggr) \, d\varrho^{\otimes n}(x), \,\, x=(x_i).
$$

\begin{theorem} \label{wasserstein}
Given the Curie-Weiss model $\widehat{P}_{n, \beta}$ and $\varrho \in \mathcal{B}$ in GHS,
let $\alpha$ be the global minimum of type $k$ and strength $\mu$ of $G_{\varrho}$. In the case $k=1$, for any
uniformly Lipschitz function $h$ we obtain for $W=S_n/\sqrt{n}$ that
$$
\big| \E \bigl( h(W) \bigr)  - \Phi_W(h) \big| \leq \|h'\| \, C \, \frac{\max \bigl( \E|X_1|^3, \E|X_1'|^3 \bigr)}{\sqrt{n}}.
$$
Here $C$ is a constant depending on $0 < \beta < 1$ and $\Phi_W(h) := \int_{\R} h(z) \Phi_W(dz)$.
The random variable $X_i'$ is drawn from the conditional distribution of the $i$'th coordinate $X_i$ given $(X_j)_{j \not= i}$
(this choice will be explained in Section 3).
For $k \geq 2$ we obtain for any uniformly Lipschitz function $h$ and for $W := \frac{S_n - n \alpha}{n^{1- 1/2k}}$
$$
\big| \E \bigl( h(W) \bigr)  - \widehat{F}_{W,k}(h) \big| \leq \|h'\|  \biggl( C_1 \frac{1}{n^{1/k}} + \frac{C_2 \, \max \bigl( \E|X_1|^3, \E|X_1'|^3 \bigr)}{n^{1-1/2k}}
\biggr).
$$
Here $C_1, C_2$ are constants, and $\widehat{F}_{W,k}(h) := \int_{\R} h(z) \widehat{F}_{W,k}(dz)$.
\end{theorem}

\begin{remark}
Assume that there exists a $\delta$ such that for any uniformly Lipschitz function $h$, $|\E h(W) - F(h) |\leq \delta \|h'\|$,
where $W$ is a random variable, $F(h) := \int_{\R} h(z) F(dz)$ for some distribution function $F$, then from the definition of the Wasserstein distance
it follows immediately that $\sup_{h \in {\rm Lip}(1)} |\E h(W) -F(h)| \leq \delta$. Moreover, the Kolmogorov distance
$\sup_{z} | P(W \leq z) - F(z)|$ can be bounded by $c_F \, \delta^{1/2}$, where $c_F$ is some constant depending on $F$ (the proof 
follows the lines of \cite[Theorem 3.1]{ChenShao:2005}).
\end{remark}

\begin{remark} \label{lebo}
In \cite{Ellis/Monroe/Newman:1976}, the distribution of the spins $\varrho$ are allowed to depend on the site.
They define a subclass ${\mathcal G}$ of $\mathcal{B}$ such that for $\varrho_1, \ldots, \varrho_n \in {\mathcal G}$ the GHS
inequality holds. In Section 7 we present a large class of measures which belong to $\mathcal{G}$
(see \cite[Theorem 1.2]{Ellis/Monroe/Newman:1976}). The GHS inequality itself has a number of interesting implications like
the concavity of the average magnetization as a function of the external field $h$ or the monotonicity of correlation length in Ising
models. These and other implications can be found in  \cite{Ellis/Monroe/Newman:1976} and references therein. Note that
for $\varrho \in {\rm GHS}$, $\phi_{\varrho}(s) \leq \frac 12 \sigma_{\varrho}^2 s^2$ for all real $s$, where $\sigma_{\varrho}^2 = \int_{\R} x^2 \, \varrho(dx)$.
These measures are called {\it sub-Gaussian}. Very important for our proofs of Berry-Esseen bounds will be the following
correlation-inequality due to Lebowitz \cite{Lebowitz:1974}: If $\E$ denotes the expectation with respect to the measure $P_{n, \beta, h_1, \ldots, h_n}$,
one observes easily that for any $\varrho \in {\mathcal B}$ and sites $i,j,k,l \in \{1, \ldots, n\}$ the following identity holds:
\begin{eqnarray} \label{ursell4}
& & \frac{\partial^3}{\partial h_i \, \partial h_j \, \partial h_k} \E (X_l) \biggl|_{{\rm all} \,\, h_i =0} \\
& = & \E(X_i X_j X_k X_l) - \E(X_i X_j) \E(X_k X_l)  - \E(X_i X_k) \E(X_j X_l)  - \E(X_i X_l) \E(X_j X_k). \nonumber
\end{eqnarray}
Lebowitz \cite{Lebowitz:1974} proved that if $\varrho \in {\rm GHS}$, then \eqref{ursell4} is non-positive (see \cite[V.13.7.(b)]{Ellis:LargeDeviations}
and \cite{Kondo/Otofuji/Sugiyama:1985}). Stein's method reduces to the computation of, or bounds on, {\it low order} moments, perhaps
even only on variances of certain quantities. Such variance computations can be very difficult. We will see in the proof
of Theorem \ref{CWgeneral} and Theorem \ref{CWgendep} the use of Lebowitz' inequality for bounding the variances successfully.
\end{remark}
\medskip

In the situation of Theorem \ref{CWgeneral} and Theorem \ref{CWgendep} we can bound higher order moments as follows:

\begin{lemma} \label{momentsgeneral}
Given $\varrho \in {\mathcal B}$, let $\alpha$ be one of the global minima of maximal type $k$ for $k \geq 1$ and strength $\mu$ of $G_{\varrho}$. For
$$
W :=  \frac{S_n - n \alpha}{n^{1 - 1/2k}}
$$
we obtain for any $l \in \N$
\begin{equation*}
\E |W|^{l} \leq \const(l).
\end{equation*}
\end{lemma}

We prepare for the proof of Lemma \ref{momentsgeneral}. It considers a well known transformation --
sometimes called the {\it Hubbard--Stratonovich transformation} -- of our measure of interest.

\begin{lemma}\label{hubbard}
Let $m\in \R$ and $0<\gamma<1$ be real numbers.
Consider the measure $Q_{n, \beta}:= \bigl( P_{n}\circ \left(\frac{S_n-nm}{n^\gamma}\right)^{-1} \bigr) \ast
\mathcal{N}(0,\frac 1 {\beta n^{2\gamma-1}})$ where $\mathcal{N}(0,\frac 1 {\beta n^{2\gamma-1}})$ denotes
a Gaussian random variable with mean zero and variance $\frac 1 {\beta n^{2\gamma-1}}$. Then
for all $n \ge 1$ the measure $Q_{n, \beta}$ is absolutely continuous with density
\begin{equation}\label{density}
\frac{\exp\left(-n G(\frac s {n^{1-\gamma}}+m)\right)}{\int_{\R} \exp\left(-n G(\frac s {n^{1-\gamma}}+m)\right)ds},
\end{equation}
where $G$ is defined in equation \eqref{Gdef}.
\end{lemma}

\begin{remark}
As shown in \cite{Ellis/Newman:1978}, Lemma 3.1, our condition \eqref{condrho} ensures that
$$
\int_{\R} \exp\left(-n G \left( \frac s {n^{1-\gamma}}+m \right) \right)ds
$$
is finite, such that the above density is well defined.
\end{remark}

\begin{proof}[Proof of Lemma \ref{hubbard}]
The proof of this lemma can be found at many places, e.g. in \cite{Ellis/Newman:1978}, Lemma 3.3.
\end{proof}

\begin{proof}[Proof of Lemma \ref{momentsgeneral}]
We apply the Hubbard-Stratonovich transformation with $\gamma=1-1/2k$. It is clear that this does not change the finiteness of any of the moments of $W$.
Using the Taylor expansion \eqref{Taylor} of $G$, we see that the density of $Q_{n, \beta}$ with respect to Lebesgue measure is given by
$\mathrm{Const.}\exp(-x^{2k})$ (up to negligible terms, see e.g. \cite{Ellis/Newman:1978}, \cite{Eichelsbacher/Loewe:2004}).
A measure with this density, of course, has moments of any finite order.
\end{proof}

\begin{remark} As we will see, we only have to bound $\E(W^4)$ in the classical model, when $0< \beta <1$. This can be obtained
directly using the definition of $P_n$ and Taylor-expansion. But already for the classical model, for $\beta=1$, it is quite cumbersome
to bound higher order moments via direct calculations.
\end{remark}

\medskip
In Section 2, we develop in Theorem \ref{ourmain}, Corollary \ref{corsigma} and Corollary \ref{corsigma2} refinements of
Stein's method for exchangeable pairs in the case of normal approximation.
As a first application we prove Theorem \ref{CW} in Section 3. In Section 4 we develop Stein's method for exchangeable
pairs for a rich class of other distributional approximations. Obtaining good bounds for the solutions of the
corresponding Stein equations in the appendix, we prove Theorem \ref{CWcritical} and Theorem \ref{CWbetadep} in Section 5, applying
Theorem \ref{generaldensity}.
In Section 6, we proof Theorems \ref{CWgeneral}, \ref{CWgendep} and \ref{wasserstein}, applying Corollary \ref{corsigma2} and
Theorem \ref{generaldensity2}.
Section 7 contains a collection of examples including the Curie-Weiss model with three states, studying liquid helium, and
a continuous Curie-Weiss model, where the single spin distribution $\varrho$ is a uniform distribution.

\section{Stein's method with exchangeable pairs for normal approximation}\label{hajek}

Stein introduced in \cite{Stein:1986} the exchangeable pair approach. Given a random variable $W$, Stein's method
is based on the construction of another variable $W'$ (some coupling) such that the pair $(W, W')$ is exchangeable, i.e. their
joint distribution is symmetric. The approach essentially uses the elementary fact that if $(W,W')$ is an exchangeable
pair, then $\E g(W,W')=0$ for all antisymmetric measurable functions $g(x,y)$ such that the expectation exists.
A theorem of Stein (\cite[Theorem 1, Lecture III]{Stein:1986}) 
shows that a measure of proximity of $W$ to normality may be provided in terms
of the exchangeable pair, requiring $W'-W$ to be sufficiently small. He assumed the linear regression property
$$
\E(W'|W)=(1-\lambda) \, W
$$
for some $0 < \lambda <1$. This approach has been successfully applied in many models, see \cite{Stein:1986} and
for example \cite{DiaconisStein:2004} and references therein. In \cite{Rinott/Rotar:1997}, the range of application was
extended by replacing the linear regression property by a weaker condition, allowing to hold the regression property only
approximately.
The exchangeable pair approach is also successful for other distributional approximations, as will be shown in Section 4.
We develop Stein's method by replacing the linear regression property by
$$
\E(W'|W) = W + \lambda \, \psi(W) + R(W),
$$
where $\psi(x)$ will be depend on a continuous distribution under consideration. Before we consider in this section the case of normal approximation,
we mention that this is not the first paper to study other distributional approximations via Stein's method. For a rather large
class of continuous distributions, the Stein characterization was introduced in \cite{DiaconisStein:2004}, following \cite[Chapter 6]{Stein:1986}.
In \cite{DiaconisStein:2004}, the method of exchangeable pairs was introduced for this class of distribution and used in a simulation
context. Recently, the exchangeable pair approach was introduced for exponential approximation in \cite[Lemma 2.1]{Chatterjee/Fulman/Roellin:2009}.

For measuring the distance of the distribution of $W$ and the
standard normal distribution (or any other distribution), we would like to bound
$$
| \E h(W) - \Phi(h) |
$$
for a class of test functions $h \in \Hcal$, where $\Phi(h) := \int_{-\infty}^{\infty} h(z) \Phi(dz)$ and $\Phi$ is the standard normal distribution function.
One advantage of Stein's method is that we are able to obtain bounds for different distances like the Wasserstein distance $d_{\rm{w}}$, 
the total variation distance  $d_{\rm{TV}}$ or the Kolmogorov distance $d_{\rm{K}}$.
In \cite{Rinott/Rotar:1997}, the exchangeable pair
approach of Stein was developed for a broad class of non smooth functions $h$, applying standard smoothing inequalities.
%Following \cite{Rinott/Rotar:1997}, we define for a given function $h : \R \to \R$ and $\varepsilon >0$
%$$
%h_{\varepsilon}^+(x) = \sup \{ h(x+y): |y| \leq \varepsilon \}, \quad h_{\varepsilon}^-(x) = \inf \{ h(x+y): |y| \leq \varepsilon \},
%$$
%and
%$$
%\tilde{h}(x, \varepsilon) = h_{\varepsilon}^+(x) - h_{\varepsilon}^-(x).
%$$
%Let $\Hcal$ be a class of measurable functions on a real line that the following hold:

%\begin{enumerate}
%\item{All functions in $\Hcal$ are uniformly bounded in absolute value by a constant assumed to be 1 without loss of generality.}
%\item{For any real numbers $c$ and $d$ and for any $h \in \Hcal$, the function $h(cx+d)$ belongs to $\Hcal$.}
%\item{For any $\varepsilon$ and any $h \in \Hcal$, the functions $h_{\varepsilon}^+,  h_{\varepsilon}^-$ are also in $\Hcal$, and
%$$
%\int \tilde{h}(x, \varepsilon) \Phi(dx) \leq a \, \varepsilon
%$$
%for some constant $a$ which depends only on the class $\Hcal$.}
%\end{enumerate}
%The indicators of all half lines compose a class which satisfies these conditions with $a=\sqrt{2/\pi}$ (the
%indicators of all intervalls as well with $a =2 \sqrt{2/ \pi}$).
They proved the following:

%To prove Theorem \ref{CW}, we will improve the following theorems due to Rinott and Rotar \cite[Theorem 1.2]{Rinott/Rotar:1997}
%and due to Shao and Su \cite[Theorem 2.1]{Shao/Sue:2005}.

\begin{theorem}[Rinott, Rotar: 1997]\label{RR} Consider a random variable $W$ with $\E(W)=0$ and $\E(W^2)=1$. Let $(W, W')$ be an exchangeable pair
(i.e., their joint distribution is symmetric). Define a random variable $R=R(W)$ by
$$
E(W' |W) = (1 - \lambda) W +R,
$$
where $\lambda$ is a number satisfying $0 < \lambda <1$. If moreover
$$
|W' - W| \leq A
$$
for a constant $A$. Then
%%%%%with
%$$
%\delta := \sup \{ | \E h(W) - \Phi(h)| : h \in \Hcal \}
%$$
one obtains
\begin{equation} \label{rr}
\sup_{z \in \R} |P(W \leq z) - \Phi(z)|  \leq \frac{12}{\lambda} \sqrt{ {\rm  var} \{ \E[(W'-W)^2 |W] \}} + 37 \frac{\sqrt{\E(R^2)}}{\lambda}
+ 48 \sqrt{2/ \pi} \frac{A^3}{\lambda} +  \sqrt{2/ \pi} \frac{A^2}{\sqrt{\lambda}}.
\end{equation}
\end{theorem}

\begin{remark} Rinott and Rotar also proved a bound in the case, where $|W'-W|$ is not assumed to be bounded. In this case, the last two
summands on the right hand side of \eqref{rr} have to be replaced by
$$
\sqrt{ \frac{a}{\lambda} \E | W'-W|^3}.
$$
This estimation is crude, since even for a normalized sum of $n$ independent variables $W$, it leads to a bound of the order $n^{-1/4}$.
The advantage of the results in \cite{Rinott/Rotar:1997} is, that these bounds do not only apply to indicators on half lines, but also to
a broad class of non smooth test functions, see \cite[Section 1.2]{Rinott/Rotar:1997}.
\end{remark}

Chen and Shao introduced a concentration inequality approach. Here a concentration inequality is proved using the Stein
identity (see \cite{ChenShao:2001} and \cite{ChenShao:2005}). In the context of the construction of an exchangeable pair,
in \cite{ShaoSu:2006} Shao and Su proved the following theorem:

\begin{theorem}[Shao, Su: 2005]\label{SS} Let $W$ be a random variable with $\E(W)=0$ and $\E(W^2) \leq 1$ and $(W, W')$ be an exchangeable pair
such that
$$
E(W' |W) = (1 - \lambda) W
$$
with  $0 < \lambda <1$, then for any $a >0$
\begin{equation} \label{ss1}
\sup_{z \in \R} |P(W \leq z) - \Phi(z) | \leq \sqrt{ \E \biggl( 1 - \frac{1}{2 \lambda} E((W-W')^2|W) \biggr)^2} + \frac{0.41 a^3}{\lambda} +
1.5 a + \frac{1}{2 \lambda} \E ( (W-W') 1_{\{|W-W'| \geq a\}} ).
\end{equation}
If $ |W - W'| \leq A$, then the bound reduces to
\begin{equation} \label{ss2}
\sup_{z \in \R} |P(W \leq z) - \Phi(z) | \leq \sqrt{ \E \biggl( 1 - \frac{1}{2 \lambda} E((W-W')^2|W) \biggr)^2} + \frac{0.41 A^3}{\lambda} + 1.5 A.
\end{equation}
\end{theorem}

\begin{remark}
When $|W-W'|$ is bounded, \eqref{ss2} improves \eqref{rr} with respect to the constants.
\end{remark}

Following the lines of the proofs in \cite{Rinott/Rotar:1997} and \cite{ShaoSu:2006}, we obtain the following refinement:
Given two random variables $X$ and $Y$ defined on a common probability space, we denote by
$$
d_{\rm{K}}(X,Y) := \sup_{z \in \R} | P(X \leq z) - P(Y \leq z)|
$$
the Kolmogorov distance of the distributions of $X$ and $Y$.

\begin{theorem} \label{ourmain}
Let $(W, W')$ be an exchangeable pair of real-valued random variables such that
$$
E(W' |W) = (1 - \lambda) W +R
$$
for some random variable $R = R(W)$ and with $0 < \lambda <1$. Assume that $\E(W^2) \leq 1$. Let $Z$ be a random variable
with standard normal distribution.
Then for any $A >0$,
\begin{eqnarray*}
d_{\rm{K}}(W, Z) & \leq & \sqrt{ \E \biggl( 1 - \frac{1}{2 \lambda}  \E[(W'-W)^2 |W] \biggr)^2} +
\biggl( \frac{\sqrt{2 \pi}}{4} + 1.5 A \biggr) \frac{\sqrt{\E(R^2)}}{\lambda} \\
& &+ \frac{0.41 A^3}{\lambda} +  1.5 A  + \frac{1}{2 \lambda} \E \bigl( (W-W')^2 1_{\{|W-W'| \geq A\}} \bigr).
\end{eqnarray*}
If $|W - W'| \leq A$ for a constant
$A$, we obtain the bound
\begin{equation} \label{mainbound}
d_{\rm{K}}(W, Z) \leq \sqrt{ \E \biggl( 1 - \frac{1}{2 \lambda}  \E[(W'-W)^2 |W] \biggr)^2}  + \biggl( \frac{\sqrt{2 \pi}}{4} + 1.5 A \biggr)
\frac{\sqrt{\E(R^2)}}{\lambda} + \frac{0.41 A^3}{\lambda} +  1.5 A.
\end{equation}
\end{theorem}

\begin{remark}
When $|W-W'|$ is bounded, \eqref{mainbound} improves \eqref{rr} with respect to the Berry-Esseen constants.
\end{remark}

\begin{proof} We sketch the proof: For a function $f$ with $|f(x)| \leq C(1 +|x|)$ we obtain
\begin{eqnarray} \label{start}
0 & = & \E \bigl( (W-W')(f(W')+f(W)) \bigr) \nonumber \\
& = & \E \bigl( (W-W')(f(W')-f(W)) \bigr) + 2 \lambda \E(W f(W)) - 2 \E(f(W) \, R).
\end{eqnarray}
Let $f=f_z$ denote the solution of the Stein equation
\begin{equation} \label{steinnormal}
f_z'(x) - xf_z(x) = 1_{\{x \leq z\}}(x) - \Phi(z).
\end{equation}
We obtain
\begin{eqnarray} \label{tis}
P(W \leq z) - \Phi(z) & = & \E (f'(W) - W f(W)) \nonumber \\
& = & \E (f'(W)) - \frac{1}{2 \lambda} \E  \bigl( (W-W')(f(W)-f(W')) \bigr) - \frac{1}{\lambda} \E (f(W) \, R) \nonumber \\
& = & \E \biggl( f'(W) \bigl( 1 - \frac{1}{2 \lambda} (W-W')^2 \bigr) \biggr) + \E \biggl( f'(W) \frac{1}{2 \lambda} (W-W')^2 \biggr) \nonumber \\
& & - \frac{1}{2 \lambda} \E  \bigl( (W-W')(f(W)-f(W')) \bigr) - \frac{1}{\lambda} \E (f(W) \, R) \nonumber \\
& =&  \E \biggl( f'(W) \bigl( 1 - \frac{1}{2 \lambda} (W-W')^2 \bigr) \biggr) - \frac{1}{2 \lambda} \E (2 f(W) \, R) \nonumber \\
& & - \frac{1}{2 \lambda} \E \biggl[ (W-W') \bigl( f(W) - f(W') - (W-W')f'(W) \bigr) \biggr] \nonumber \\
& =: & T_1 + T_2 + T_3.
\end{eqnarray}
Using $|f'(x)| \leq 1$ for all real $x$ (see \cite[Lemma 2.2]{ChenShao:2005}), we obtain the bound
$$
|T_1| \leq \sqrt{ \E \biggl( 1 - \frac{1}{2 \lambda}  \E[(W'-W)^2 |W] \biggr)^2}.
$$
Using $0 < f(x) \leq \sqrt{2 \pi}/4$ (see \cite[Lemma 2.2]{ChenShao:2005}), we have
$$
|T_2| \leq \frac{\sqrt{2 \pi}}{4 \lambda} \E(|R|) \leq  \frac{\sqrt{2 \pi}}{4 \lambda} \sqrt{ \E(R^2)}.
$$
Bounding $T_3$ we apply the concentration technique, see \cite{ShaoSu:2006}:
\begin{eqnarray} \label{t3}
(- 2 \lambda) \, T_3 & = & \E \biggl( (W-W') 1_{\{|W-W'| > A\}} \int_{-(W-W')}^0 (f'(W+t)-f'(W)) dt \biggr) \nonumber \\
&+ &   \E \biggl( (W-W') 1_{\{|W-W'| \leq A\}} \int_{-(W-W')}^0 (f'(W+t)-f'(W)) dt \biggr).
\end{eqnarray}
The modulus of the first term can be bounded by $\E \bigl( (W-W')^2 1_{\{|W-W'| >A\}} \bigr)$ using $|f'(x) - f'(y)| \leq 1$
for all real $x$ and $y$ (see  \cite[Lemma 2.2]{ChenShao:2005}). Using the Stein identity \eqref{steinnormal}, the second summand can be represented
as
\begin{eqnarray*}
& & \E \biggl( (W-W') 1_{\{|W-W'| \leq A\}} \int_{-(W-W')}^0 \bigl( (W+t)f(W+t) - Wf(W) \bigr) dt \biggr) \\
& & + \E \biggl( (W-W') 1_{\{|W-W'| \leq A\}} \int_{-(W-W')}^0 (1_{\{W+t \leq z\}} - 1_{\{W \leq z\}}) dt \biggr) =: U_1 + U_2.
\end{eqnarray*}
Next observe that $|U_1| \leq 0.82 A^3$, see \cite{ShaoSu:2006}: by the mean value theorem one gets
$$
(W+t) f(W+t) - W f(W) = W ( f(W+t) - f(W)) + t f(W+t) = W \bigl( \int_0^1 f'(W + ut) t du \bigr) + t f(W+t).
$$
Hence
$$
|(W+t) f(W+t) - W f(W) | \leq |W| \, |t|  + |t| \sqrt{2 \pi}/4 = |t| (\sqrt{2 \pi}/4 + |W|).
$$
Using $\E |W| \leq \sqrt{ \E(W^2)} \leq 1$ gives the bound. The term $U_2$ can be bounded by
$$
\E \bigl( (W-W')^2 I_{\{0 \leq (W-W') \leq A\}} \, 1_{\{z \leq W \leq z+A\}} \bigr).
$$
Under the assumptions of our Theorem we proceed as in \cite{ShaoSu:2006} and obtain the following
concentration inequality:
\begin{equation} \label{conz1}
\E \bigl( (W-W')^2 I_{\{0 \leq (W-W') \leq a\}} \, 1_{\{z \leq W \leq z+A\}} \bigr) \leq 3A (\lambda + \E(R)).
\end{equation}
To see this, we apply the estimate
$$
\E \bigl( (W-W')^2 I_{0 \leq (W-W') \leq A} \, 1_{z \leq W \leq z+A} \bigr) \leq \E \bigl( (W-W')(f(W)-f(W')) \bigr),
$$
see \cite{ShaoSu:2006}; here $f$ is defined by $f(x):= -1.5 A$ for $x \leq z-A$, $f(x):=1.5 A$ for $x \geq z+2A$ and $f(x):=x-z-A/2$
in between. Now we apply \eqref{start} and get
$$
\E \bigl( (W-W')^2 I_{ \{0 \leq (W-W') \leq A\}}\, 1_{\{z \leq W \leq z+A\}} \bigr) \leq 2 \lambda \E(Wf(W)) + 2 \E(f(W)R) \leq 3A(\lambda + \E(|R|)),
$$
where we used $\E(|W|) \leq \sqrt{ \E(W^2)} \leq 1$. Similarly, we obtain
$$
U_2 \geq -  3A(\lambda + \E(R)).
$$
\end{proof}

\begin{remark}
In Theorem \ref{ourmain}, we assumed $\E(W^2) \leq 1$. Alternatively, let us assume that $\E(W^2)$ is finite. Then the
proof of Theorem \ref{ourmain} shows, that the third and the fourth summand of the bound \eqref{mainbound} change to
$$
\frac{A^3}{\lambda} \bigl( \frac{\sqrt{2 \pi}}{16} + \frac{\sqrt{ \E(W^2)}}{4} \bigr) + 1.5 A \, \E(|W|).
$$
\end{remark}

In the following corollary, we discuss the Kolmogorov-distance of the distribution of a random variable $W$ to a
random variable distributed according to $N(0, \sigma^2)$, the normal distribution with mean zero and variance $\sigma^2$.

\begin{cor} \label{corsigma}
Let $\sigma^2 >0$ and  $(W, W')$ be an exchangeable pair of real-valued random variables such that
\begin{equation} \label{3.2}
E(W' |W) = \bigl(1 - \frac{\lambda}{\sigma^2} \bigr) W +R
\end{equation}
for some random variable $R = R(W)$ and with $0 < \lambda <1$. Assume that $\E(W^2)$ is finite. Let $Z_{\sigma}$ be a random variable
distributed according to $N(0, \sigma^2)$. If $|W - W'| \leq A$ for a constant $A$, we obtain the bound
\begin{eqnarray} \label{mainbound2}
d_{\rm{K}}(W, Z_{\sigma}) & \leq & \sqrt{ \E \biggl( 1 - \frac{1}{2 \lambda}  \E[(W'-W)^2 |W] \biggr)^2}  +
\biggl( \frac{\sigma \sqrt{2 \pi}}{4} + 1.5 A \biggr) \frac{\sqrt{\E(R^2)}}{\lambda} \nonumber \\
& + & \frac{A^3}{\lambda} \biggl(\frac{\sqrt{2 \pi \sigma^2}}{16} + \frac{\sqrt{\E(W^2)}}{4} \biggr) +  1.5 A \sqrt{\E(W^2)}.
\end{eqnarray}
\end{cor}

\begin{proof}
Let us denote by $f_{\sigma} := f_{\sigma,z}$ the solution of the Stein equation
\begin{equation} \label{stein7}
f_{\sigma,z}'(x) - \frac{x}{\sigma^2} f_{\sigma,z}(x) = 1_{\{x \leq z\}}(x) - F_{\sigma}(z)
\end{equation}
with $F_{\sigma}(z) := \frac{1}{\sqrt{2 \pi} \sigma} \int_{- \infty}^z \exp \bigl( - \frac{y^2}{2 \sigma^2} \bigr) \, dy$.
It is easy to see that the identity $f_{\sigma, z}(x) = \sigma f_{z} \bigl( \frac{x}{\sigma} \bigr)$, where $f_z$ is the solution
of the corresponding Stein equation of the standard normal distribution, holds true. Using \cite[Lemma 2.2]{ChenShao:2005} we obtain
$0 < f_{\sigma} (x) < \sigma \frac{\sqrt{2 \pi}}{4}$, $|f_{\sigma}'(x) | \leq 1$, and $|f_{\sigma}'(x) - f_{\sigma}'(y)| \leq 1$.
With \eqref{3.2} we arrive at
$$
P(W \leq z ) - F_{\sigma}(z) = T_1 + T_2 + T_3
$$
with $T_i$'s defined in \eqref{tis}. Using the bounds of $f_{\sigma}$ and $f_{\sigma}'$, the bound of $T_1$ is the
same as in the proof of Theorem \ref{ourmain}, whereas the bound of $T_2$ changes to
$$
|T_2| \leq \sigma \frac{\sqrt{2 \pi}}{4 \lambda} \sqrt{\E(R^2)}.
$$
Since we consider the case $|W-W'| \leq A$, we have to bound
$$
T_3 = -\frac{1}{2 \lambda}  \E \biggl( (W-W') 1_{\{|W-W'| \leq A\}} \int_{-(W-W')}^0 (f'(W+t)-f'(W)) dt \biggr).
$$
Using the Stein identity \eqref{stein7}, the mean value theorem as well as the concentration inequality-argument along the
lines of the proof of Theorem \ref{ourmain}, we obtain
$$
|T_3| \leq \frac{A^3}{\lambda} \bigl( \frac{\sqrt{\E(W^2)}}{4} + \frac{\sigma \sqrt{2 \pi}}{16} \bigr) + 1.5 A \bigl(
\sqrt{\E(W^2)} + \frac{\sqrt{\E(R^2)}}{\lambda} \bigr).
$$
Hence the corollary is proved.
\end{proof}

With \eqref{3.2} we obtain $\E(W-W')^2 = \frac{2 \lambda}{\sigma^2} \E(W^2) - 2 \E(W\, R)$. Therefore
\begin{equation} \label{hihi}
\E \biggl(1-  \frac{1}{2 \lambda}  \E[(W'-W)^2 |W] \biggr) = 1 - \frac{\E(W^2)}{\sigma^2} + \frac{\E(W \, R)}{\lambda},
\end{equation}
so that the bound in Corollary \ref{corsigma} is only useful when $\E(W^2)$ is close to $\sigma^2$ (and
$\E(W\,R)/ \lambda$ is small). An alternative
bound can be obtained comparing with a $N(0, \E(W^2))$-distribution.

\begin{cor} \label{corsigma2}
In the situation of Corollary \ref{corsigma}, let $Z_{W}$ denote the $N(0, \E(W^2))$ distribution. We obtain
\begin{eqnarray} \label{mainbound3}
d_{\rm{K}}(W, Z_{W}) & \leq & \frac{\sigma^2}{2 \lambda}  \bigl( {\rm Var} \bigl( \E[(W'-W)^2 |W] \bigr) \bigr)^{1/2}  +
\sigma^2 \biggl( \frac{\sqrt{\E(W^2)}\,  \sqrt{2 \pi}}{4} + 1.5 A \biggr) \frac{\sqrt{\E(R^2)}}{\lambda} \nonumber \\
& & \hspace{-2cm} + \sigma^2 \, \frac{A^3}{\lambda} \biggl(\frac{\sqrt{\E(W^2)} \, \sqrt{2 \pi}}{16} + \frac{\sqrt{\E(W^2)}}{4} \biggr)
+  \sigma^2 \, 1.5 A \, \sqrt{\E(W^2)} + \sigma^2 \frac{\sqrt{\E(W^2)} \, \sqrt{\E(R^2)}}{\lambda}.
\end{eqnarray}
\end{cor}

\begin{proof}
With \eqref{hihi} we get $\E(W^2)= \sigma^2 \bigl( \frac{1}{2 \lambda} ( E(W-W')^2 + 2 \E(W \, R)) \bigr)$.
With the definition of $T_2$ and $T_3$ as in \eqref{tis} we obtain
\begin{eqnarray} \label{tistis}
\E \bigl( \E(W^2) f'(W) - W f(W) \bigr) & = & \sigma^2 \E \biggl( \frac{\E(W-W')^2 + 2 \E(W \, R)}{2 \lambda}  f'(W) \biggr) - \E( W \, f(W)) \nonumber \\
& & \hspace{-5cm}  = \sigma^2 \E \biggl( f'(W) \biggl( \frac{\E(W-W')^2 - \E[(W-W')^2|W]}{2 \lambda} \biggr) \biggr)
+ \sigma^2 (T_2 + T_3) + \sigma^2 \frac{\E(W \, R)}{\lambda}.
\end{eqnarray}
Remark that now $\sigma^2$ in \eqref{3.2} is a parameter of the exchangeable-pair identity and no longer the parameter
of the limiting distribution. We apply \eqref{stein7} and exchange every $\sigma^2$ in \eqref{stein7} with $\E(W^2)$.
Applying Cauchy-Schwarz to the first summand and bounding the other terms as in the proof of Corollary \ref{corsigma} leads
to the result.
\end{proof}

\section{Berry-Esseen bounds for the classical Curie-Weiss model}

Let $\varrho$ be the symmetric Bernoulli measure and $0 < \beta <1$. Then
$$
W := W_n := \frac{1}{\sqrt{n}} \sum_{i=1}^n X_i.
$$
converges in distribution to a $N(0, \sigma^2)$ with $\sigma^2 = (1 - \beta)^{-1}$:

\begin{proof}[Proof of Theorem~\ref{CW}]
We consider the usual construction of an exchangeable pair. We produce a spin collection $X'= (X_i')_{i \geq 1}$ via
a {\it Gibbs sampling} procedure: select a  coordinate, say $i$, at random and replace $X_i$ by $X_i'$ drawn from the
conditional distribution of the $i$'th coordinate given $(X_j)_{j \not= i}$. Let $I$ be a random variable taking values $1, 2, \ldots, n$
with equal probability, and independent of all other random variables. Consider
$$
W' := W - \frac{X_I}{\sqrt{n}} +  \frac{X_I'}{\sqrt{n}} = \frac{1}{\sqrt{n}} \sum_{j \not= I} X_j + \frac{X_I'}{\sqrt{n}}.
$$
Hence $(W,W')$ is an exchangeable pair and
$$
W-W' = \frac{X_I - X_I'}{\sqrt{n}}.
$$
Let $\mathcal{F} := \sigma(X_1, \ldots, X_n)$.
Now we obtain
$$
\E[W-W'| \mathcal{F}] = \frac{1}{\sqrt{n}} \frac 1n \sum_{i=1}^n \E[ X_i - X_i' | \mathcal{F}] = \frac{1}{n} \, W -
\frac{1}{\sqrt{n}} \frac 1n \sum_{i=1}^n  \E[ X_i'|\mathcal{F}].
$$
The conditional distribution at site $i$ is given by
$$
P_n \bigl( x_i | (x_j)_{j \not= i} \bigr) = \frac{ \exp \bigl( x_i \, \beta \, m_i(x) \bigr)}{\exp \bigl( \beta m_i(x) \bigr) + \exp \bigl( - \beta m_i(x) \bigr)},
$$
with
$$
m_i(x) := \frac 1n \sum_{j \not= i} x_j, \,\, i=1, \ldots, n.
$$
%This describes the weak long-range interactions in the Curie-Weiss model. {\tt Ehrlicherweise gibt es gar keinen Range, da es keinen Abstande gibt .. 
%und ich weiß auch nichtgenau, was 'This' ist}
It follows that
$$
\E[X_i' | \mathcal{F}] = \E[ X_i | (X_j)_{j \not= i}] = \tanh (\beta m_i(X)).
$$
Now
$\frac{1}{\sqrt{n}} \frac 1n \sum_{i=1}^n  \tanh (\beta m_i(X)) =
\frac{1}{\sqrt{n}} \frac 1n \sum_{i=1}^n \bigl(  \tanh (\beta m_i(X)) -  \tanh (\beta m(X)) \bigr) +
\frac{1}{\sqrt{n}}  \tanh (\beta m(X)) =: R_1 + R_2$
with $m(X) := \frac 1n \sum_{i=1}^n X_i$. Taylor-expansion $\tanh(x) = x + \mathcal{O}(x^3)$
leads to
\begin{equation*}
R_2 =\frac{1}{\sqrt{n}} \beta m(X) + \frac{1}{\sqrt{n}} \mathcal{O} \bigl( m(X)^3 \bigr)
= \frac{\beta}{n} W + \mathcal{O} \bigl( \frac{W^3}{n^2} \bigr).
\end{equation*}
Hence
\begin{equation} \label{exid1}
\E[W-W'|W] = \frac{1-\beta}{n} \, W + R = \frac{\lambda}{\sigma^2} \, W + R
\end{equation}
with $\lambda := \frac 1n$, $\sigma^2 := (1-\beta)^{-1}$ and
$R := \mathcal{O} \bigl( \frac{W^3}{n^2} \bigr) - R_1$. Since $|W-W'| = \bigl| \frac{X_I - X_I'}{\sqrt{n}} \bigr| \leq \frac{1}{\sqrt{n}} =:A$,
we are able to apply Corollary \ref{corsigma}. From Lemma \ref{momentsgeneral}
we know that for $\varrho$ being the symmetric Bernoulli distribution and for $0 < \beta < 1$ we have
$\E (W^4) \leq \rm{const.}$.
Applying this it follows that the fourth term in \eqref{mainbound2} can be bounded by
$1.5 A \frac{\sqrt{ \E(W^2)}}{\sigma^2} \leq \frac{(1-\beta) \rm{const.}}{\sqrt{n}}$,
and the third summand in \eqref{mainbound2} can be estimated as follows:
$$
\frac{A^3}{\lambda} \biggl( \frac{\sqrt{2 \pi}}{16} \sqrt{(1-\beta)} + \frac{\rm{const.}}{4} (1-\beta) \biggr) \leq \frac{1}{\sqrt{n}} \sqrt{(1-\beta)} \rm{const.}.
$$
Moreover we obtain $\E|R| \leq \E |R_1| + \mathcal{O}  \bigl( \frac{\E|W^3|}{n^2} \bigr)$.
Since $\tanh(x)$ is 1-Lipschitz we obtain $|R_1| \leq \frac{1}{\sqrt{n}} |m_i(X) - m(X)| \leq \frac{1}{n^{3/2}}$.
Therefore, with Lemma \ref{momentsgeneral}, we get $\E |R| = \mathcal{O} \bigl( \frac{1}{n^{3/2}} \bigr)$
and thus, the second summand in \eqref{mainbound2} can be bounded by
$$
\rm{const.} \biggl( \frac{\sqrt{ 2 \pi}}{4 \sqrt{(1-\beta)}} + 1.5 \frac{1}{\sqrt{n}} \biggr) \frac{1}{\sqrt{n}} = \mathcal{O} \bigl(  \frac{1}{\sqrt{n}} \bigr).
$$
To bound the first summand in \eqref{mainbound2}, we obtain
$(W-W')^2 = \frac{X_I^2}{n} - \frac{2 X_I \, X_I'}{n} + \frac{X_I'}{n}$.
Hence
$$
\E \bigl[ (W-W')^2 | \mathcal{F} \bigr] = \frac{2}{n} - \frac{2}{n^2} \sum_{i=1}^n X_i \, \tanh (\beta m_i(X)),
$$
and therefore
\begin{eqnarray*}
1 - \frac{1}{2 \lambda} \E \bigl[ (W-W')^2 | \mathcal{F} \bigr] & = & \frac 1n \sum_{i=1}^n X_i \, \tanh (\beta m_i(X)) \\
& = & \frac 1n \sum_{i=1}^n X_i \bigl( \tanh (\beta m_i(X)) - \tanh (\beta m(X)) \bigr) + m(X) \,  \tanh (\beta m(X)) \\
& =:& R_1 + R_2.
\end{eqnarray*}
By Taylor expansion we get
$R_2 = \frac{\beta}{n} W^2 + \mathcal{O} \bigl( \frac{W^4}{n^2} \bigr)$
and using Lemma \ref{momentsgeneral} we obtain $\E |R_2| = \mathcal{O} (n^{-1})$. Since $\tanh(x)$
is 1-Lipschitz we obtain $|R_1| \leq \frac 1n$. Hence
$\E |R_1 + R_2| = \mathcal{O}(n^{-1})$
and Theorem \ref{CW} is proved.
\end{proof}
\medskip

Now we discuss the critical case $\beta=1$, when $\varrho$ is the symmetric Bernoulli distribution.
For $\beta=1$, using the Taylor expansion $\tanh(x) = x - x^3/3 + \mathcal{O}(x^5)$, \eqref{exid1} would lead to
$$
\E [W-W'| W] = \frac{W^3}{3} \frac{1}{n^2} + \tilde{R}
$$
for some $\tilde{R}$. Hence it is no longer possible to apply Corollary \ref{corsigma}.
Moreover the prefactor $\lambda:=\frac{1}{n^2}$ would give growing bounds. In other words, the criticality of
the temperature value $1 / \beta_c =1$ can also be recognized by Stein's method. We already know that at the critical value,
the sum of the spin-variables has to be rescaled. Let us now define
\begin{equation} \label{w2}
W := \frac{1}{n^{3/4}} \sum_{i=1}^n X_i.
\end{equation}
Constructing the exchangeable pair $(W,W')$ in the same manner as before we will obtain
\begin{equation} \label{beta1}
\E[W-W' | W] = \frac{1}{n^{3/2}} \, \frac{W^3}{3} + R(W) =: -\lambda \psi(W) +R(W).
\end{equation}
with $\lambda=\frac{1}{n^{3/2}}$ and a reminder $R(W)$ presented later. Considering the density $p(x) = C \, \exp(-x^4/12)$, we have
$$
\frac{p'(x)}{p(x)}= \psi(x).
$$
This is the starting point for developing Stein's method for limiting distributions with a regular Lebesgue-density $p(\cdot)$
and an exchangeable pair $(W, W')$ which satisfies the condition
$$
\E [W-W'|W] = - \lambda \psi(W) + R(W) = -\lambda \frac{p'(W)}{p(W)} +R(W)
$$
with $0 < \lambda <1$. To prove \eqref{beta1}, observe that
$$
\E[W-W' | W] = \frac 1n W - \frac{1}{n^{3/4}} \frac 1n \sum_{i=1}^n \tanh( m_i(X)).
$$
By Taylor expansion and the identity $m_i(X) = m(X) - \frac{X_i}{n}$ we obtain
$$
\frac{1}{n^{3/4}} \frac 1n \sum_{i=1}^n \tanh( m_i(X)) = \frac 1n W - \frac{1}{n^{3/2}} \frac{W^3}{3} - R(W)
$$
with $R(W)$ such that $\E|R(W)| = \mathcal{O}(n^{-2})$.
The exact form of $R(W)$ will be presented in Section 5.

\section{The exchangeable pair approach for distributional approximations}
Motivated by the classical Curie-Weiss model at the critical temperature, we will develop Stein's method with the help
of exchangeable pairs as follows.
For a rather large class of continuous distributions, the Stein characterization was introduced in \cite{DiaconisStein:2004},
following the lines of \cite[Chapter 6]{Stein:1986}. The densities occurring as limit laws in models of statistical mechanics
belong to this class. Let $I$ be a real interval, where $-\infty \leq a < b \leq \infty$. A function is called {\it regular}
if $f$ is finite on $I$ and, at any interior point of $I$, $f$ possesses a right-hand limit and a left-hand limit. Further, $f$ possesses
a right-hand limit $f(a+)$ at the point $a$ and a left-hand limit $f(b-)$ at the point $b$.

Let us assume, that the regular density $p$ satisfies the following condition:
\medskip

{\bf Assumption (D)}
Let $p$ be a regular, strictly positive density on an interval $I=[a,b]$. Suppose $p$ has a derivative $p'$ that is
regular on $I$ and has only countably many sign changes and being continuous at the sign changes. Suppose moreover
that $\int_I p(x) | \log(p(x))| \, dx < \infty$
and assume that
\begin{equation} \label{psi}
\psi(x) := \frac{p'(x)}{p(x)}
\end{equation}
is regular.
\medskip

In \cite[Proposition]{DiaconisStein:2004} it is proved, that a random variable $Z$ is distributed according to the density $p$ if and only if
$$
\E \bigl( f'(Z) + \psi(Z) \, f(Z) \bigr) = f(b-) \, p(b-) - f(a+) \, p(a+)
$$
for a suitably chosen class $\mathcal{F}$ of functions $f$. The proof is integration by parts.
The corresponding Stein identity is
\begin{equation} \label{steinid2}
f'(x) + \psi(x) \, f(x) = h(x) - P(h),
\end{equation}
where $h$ is a measurable function for which $\int_I |h(x)|\, p(x) \, dx < \infty$, $P(x) := \int_{-\infty}^x p(y) \, dy$
and $P(h) := \int_I h(y) \, p(y)\, dy$.
The solution $f:=f_h$ of this differential equation is given by
\begin{equation} \label{solution}
f(x) = \frac{ \int_a^x \bigl( h(y) - Ph) \, p(y) \, dy}{p(x)}.
\end{equation}
For the function $h(x) := 1_{\{x \leq z\}}(x)$ let $f_z$ be the corresponding solution of \eqref{steinid2}.
We will make the following assumptions:
\medskip

{\bf Assumption (B1)}
Let $p$ be a density fulfilling Assumption (D). %\ref{density}.
We assume that for any absolute continuous function $h$, the solution $f_h$ of \eqref{steinid2}
satisfies
$$
\| f_h \| \leq c_1 \|h'\|, \quad \|f_h'\| \leq c_2 \|h'\| \quad \text{and} \quad \|f_h''(x)\|  \leq c_3 \|h'\|,
$$
where $c_1, c_2$ and $c_3$ are constants.
\medskip

{\bf Assumption (B2)}
Let $p$ be a density fulfilling Assumption (D) %\ref{density}.
We assume that the solution $f_z$ of
\begin{equation} \label{steinid3}
f_z'(x) + \psi(x) \, f_z(x) = 1_{\{x \leq z\}}(x) - P(z)
\end{equation}
satisfies
$$
|f_z(x)| \leq d_1, \quad |f_z'(x)| \leq d_2 \quad \text{and} \quad
|f_z'(x)-f_z'(y) | \leq d_3
$$
and
\begin{equation} \label{addcond}
|(\psi(x) \, f_z(x))'| = \bigl| ( \frac{p'(x)}{p(x)} \, f_z(x))' \bigr| \leq d_4
\end{equation}
for all real $x$ and $y$, where $d_1, d_2, d_3$ and $d_4$ are constants.
\medskip

At first glance, Condition \eqref{addcond} seem to be a rather strong or at least a
rather technical condition.

\begin{remark}
In the case of the normal approximation, $\psi(x) = -x$,
we have to bound $ (x f_z(x))'$ for the solution $f_z$ of the classical Stein equation. But it is
easy to observe that $|(x f_z'(x))'| \leq 2$ by direct calculation (see \cite[Proof of Lemma 6.5]{ChenShao:2005}).
However, in the normal approximation case, this bound would lead to a worse Berry-Esseen constant (compare Theorem \ref{ourmain}
with Theorem \ref{generaldensity}). Hence in this case we only use $d_2=d_3=1$ and $d_1 = \sqrt{2 \pi}/4$.
\end{remark}

We will see, that for all distributions appearing as limit laws in our class of Curie-Weiss models,
Condition \eqref{addcond} can be proved:

\begin{lemma} \label{genbound}
The densities $f_{k, \mu, \beta}$ in \eqref{densitysigma} and \eqref{densitygen} and the densities
in Theorem \ref{CWbetadep}, Theorem \ref{CWgeneral} and Theorem \ref{CWgendep} satisfy Assumptions (D), (B1) and
(B2).
\end{lemma}

\begin{proof}
We defer the proofs to the appendix, since they only involve careful analysis.
\end{proof}

\begin{remark} \label{gibbsrem}
With respect to all densities which appear as limiting distributions in our theorems, we restrict
ourselves to bound solutions (and its derivatives) of the corresponding Stein equation characterizing 
distributions with probability densities $p$ of the form $b_k \exp(-a_k x^{2k})$. Along the lines of the 
proof of Lemma \ref{genbound}, one would be
able to present good bounds (in the sense that Assumption (B1) and (B2) are fulfilled) 
even for measures with a probability density of the form
\begin{equation} \label{gibbs}
p(x)  = b_k \exp \bigl( -a_k V(x) \bigr),
\end{equation}
where $V$ is even, twice continuously differentiable, unbounded above at infinity, $V'\not= 0$ and 
$V'$ and $1/V'$  are increasing on $[0,\infty)$. Moreover one has to assume that $\frac{V''(x)}{|V'(x)|}$ can be bounded by a constant 
for $x \geq d$ with some $d \in \R_+$. We sketch the proof in the appendix. It is remarkable, that this class of measures is a subclass of measures which
are GHS, see Section 7. A measure with density $p$ in \eqref{gibbs} is usually called a Gibbs measure. Stein's method for discrete Gibbs measures
is developed in \cite{Eichelsbacher/Reinert:2008}. Our remark might be of use applying Stein's method for some
continuous Gibbs measure approximation.
\end{remark}

\begin{remark}
In the case of comparing with an {\it exponential distribution} with parameter $\mu$, it is easy to see, that Assumption (D) and (B2)
is fulfilled, see \cite[Example 1.6]{DiaconisStein:2004} for (D) and \cite[Lemma 2.1]{Chatterjee/Fulman/Roellin:2009} for (B2).
We have $\psi(x) = - \mu$ and $\|f_z\| \leq 1$, $\|f_z'\| \leq 1$ and $\sup_{x,y \geq 0} |f_z'(x) - f_z'(y)| \leq 1$. Thus
$|(\psi(x) f_z(x))'| = \mu |f_z'(x)| \leq \mu$.
\end{remark}

\begin{remark}
From \eqref{solution} we obtain
$$
f_z(x) = \frac{(1-P(z)) \, P(x)}{p(x)} \quad \text{for} \quad x \leq z
$$
and
$$
f_z(x) = \frac{P(z)(1-P(x))}{p(x)} \quad \text{for} \quad x \geq z.
$$
Hence
$$
\psi(x) \, f_z(x) = \frac{P(x) \, p'(x)}{p^2(x)} (1-P(z)) \quad \text{for} \quad x \leq z
$$
and
$$
\psi(x) \, f_z(x) = \frac{(1-P(x)) \, p'(x)}{p^2(x)} (P(z)) \quad \text{for} \quad x \geq z.
$$
Therefore one has to bound the derivative of
$$
 \frac{P(x) \, p'(x)}{p^2(x)} \quad \text{and} \quad  \frac{(1-P(x)) \, p'(x)}{p^2(x)},
$$
respectively, to check Condition \eqref{addcond}.
\end{remark}

%With respect to the class of Curie-Weiss models, we will consider the following class of densities.

%\begin{definition}
%Let us consider a probability density of the form
%$$
%p(x) = a \, \exp \bigl( - b \, g(x) \bigr)
%$$
%with some real numbers $a$ amd $b$.
%We have $\psi(x) = - b g'(x)$. Given such a $p$, the density
%$$
%p_r(x) = c \, \exp \bigl( - b \, \frac{g(x)}{r} \bigr)
%$$
%with a given $r \in \R_+$ (and the corresponding real number $c$ to ensure that $p_r$
%is a probability density) is a corresponding {\it rescaled density}. Here $\psi(x)= - \frac{b}{r} g'(x)$.
%\end{definition}

%\begin{remark}
%With the proof of Lemma \ref{genbound} it is easy to obtain that for every density listed in Lemma \ref{genbound},
%the corresponding rescaled density satisfies Assumption (D) and (B2).
%\end{remark}

The following result is a refinement of Stein's result \cite{Stein:1986} for exchangeable pairs.

\begin{theorem} \label{generaldensity}
Let $p$ be a density fulfilling Assumption (D).
Let $(W, W')$ be an exchangeable pair of real-valued random variables such that
\begin{equation} \label{exchangepsi}
\E [W'|W] = W + \lambda \psi(W) - R(W)
\end{equation}
for some random variable $R=R(W)$,  $0 < \lambda < 1$ and $\psi$ defined
in \eqref{psi}. Then
\begin{equation} \label{zweitesmoment}
\E (W-W')^2 = - 2 \lambda \E[W \psi(W)] + 2 \E[W \, R(W)].
\end{equation}
We obtain the following assertions:
\begin{enumerate}
\item Let $Z$ be a random variable distributed according to $p$.
Under Assumption (B1), for any uniformly Lipschitz function $h$, we obtain
$$
| \E h(W) - \E h(Z) | \leq \delta \|h'\|
$$
with
$$
\delta: = c_2 \E \biggl|  1 - \frac{1}{2 \lambda} \E \bigl( (W-W')^2| W \bigr) \biggr| + \frac{c_3}{4 \lambda} \E |W-W'|^3 + \frac{c_1}{\lambda}
\sqrt{ \E (R^2)}.
$$
\item  Let $Z$ be a random variable distributed according to $p$. Under Assumption (B2), we obtain for any $A >0$
\begin{eqnarray} \label{kolall}
d_{\rm{K}}(W, Z) & \leq & d_2 \sqrt{ \E \biggl( 1 - \frac{1}{2 \lambda}  \E[(W'-W)^2 |W] \biggr)^2} +
\big( d_1 + \frac 32 A \bigr)  \frac{\sqrt{\E(R^2)}}{\lambda} \nonumber \\
& + &  \frac{1}{\lambda} \bigl( \frac{d_4 A^3}{4} \bigr) + \frac{3A}{2} \E (|\psi(W)|)
 +   \frac{d_3}{2 \lambda} \E \bigl( (W-W')^2 1_{\{|W-W'| \geq A\}} \bigr).
\end{eqnarray}
\end{enumerate}
\end{theorem}

With \eqref{zweitesmoment} we obtain
$$
\E \biggl( 1 - \frac{1}{2\lambda} \E[(W-W')^2|W] \biggr) = 1 + \E[W \psi(W) ] - \frac{\E(W \, R)}{\lambda}.
$$
Therefore the bounds in Theorem \ref{generaldensity} are unlikely to be useful unless $-\E[W \psi(W)]$ is close to 1 and $ \frac{\E(W \, R)}{\lambda}$
is small. Alternatively bounds can be obtained comparing not with a distribution given by $p$ but with a modification
which involves $\E[W \psi(W)]$. Let $p_W$ be a probability density such that a random variable $Z$ is distributed according to $p_W$ if and only
if
$$
\E \bigl( \E[W \psi(W)] \, f'(Z) + \psi(Z) \, f(Z) \bigr) =0
$$
for a suitably chosen class of functions.

\begin{theorem} \label{generaldensity2}
Let $p$ be a density fulfilling Assumption (D).
Let $(W, W')$ be an exchangeable pair of real-valued random variables such that \eqref{exchangepsi} holds.
If $Z_W$ is a random variable distributed according to $p_W$, we obtain under
(B1), for any uniformly Lipschitz function $h$ that $| \E h(W) - \E h(Z_W) | \leq \delta' \|h'\|$
with
$$
\delta': = \frac{c_2}{2 \lambda} \bigl( {\rm Var} \bigl( \E[(W-W')^2| W] \bigr) \bigr)^{1/2}
+ \frac{c_3}{4 \lambda} \E |W-W'|^3 + \frac{c_1+c_2 \sqrt{\E(W^2)}}{\lambda}
\sqrt{ \E (R^2)}.
$$
Under Assumption (B2) we obtain for any $A >0$
\begin{eqnarray} \label{kolall2}
d_{\rm{K}}(W, Z_W) & \leq & \frac{d_2}{2 \lambda} \bigl( {\rm Var} \bigl( \E [ (W-W')^2| W] \bigr)^{1/2} +
\big( d_1 + d_2 \sqrt{\E(W^2)} +\frac 32 A \bigr)  \frac{\sqrt{\E(R^2)}}{\lambda} \nonumber \\
& + &  \frac{1}{\lambda} \bigl( \frac{d_4 A^3}{4} \bigr) + \frac{3A}{2} \E (|\psi(W)|)
 +   \frac{d_3}{2 \lambda} \E \bigl( (W-W')^2 1_{\{|W-W'| \geq A\}} \bigr).
\end{eqnarray}
\end{theorem}

\begin{proof}[Proof of Theorem \ref{generaldensity}]
Interestingly enough, the proof is a quite simple adaption of the results in \cite{Stein:1986} and follows the
lines of the proof of Theorem \ref{ourmain}. For a function $f$ with $|f(x)| \leq C(1 +|x|)$ we obtain
\begin{eqnarray} \label{start2}
0 & = & \E \bigl( (W-W')(f(W')+f(W)) \bigr) \nonumber \\
& = & \E \bigl( (W-W')(f(W')-f(W)) \bigr) - 2 \lambda \E(\psi(W) \, f(W)) + 2 \E(f(W) \, R(W)),
\end{eqnarray}
which is equivalent to
\begin{equation} \label{start3}
\E(\psi(W) \, f(W)) = - \frac{1}{2 \lambda} \E \bigl( (W-W')(f(W)-f(W')) \bigr) + \frac{1}{\lambda}  \E(f(W) \, R(W))
\end{equation}

{\bf Proof of (1):}
Now let $f=f_h$ be the solution of the Stein equation \eqref{steinid2}, and define
$$
\widehat{K}(t) := (W-W') \bigl( 1_{\{-(W-W') \leq t \leq 0 \}} - 1_{\{0 < t \leq -(W-W') \}} \bigr) \geq 0.
$$
By \eqref{start3}, following the calculations on page 21 in \cite{ChenShao:2005}, we simply obtain
\begin{eqnarray*}
|\E h(W) - \E h(Z) | & = & |\E \bigl( f'(W) + \psi(W) \, f(W) \bigr) |\\
& = & \bigl| \E \bigl( f'(W) \biggl( 1 - \frac{1}{2 \lambda} (W-W')^2 \biggr) + \frac{1}{2 \lambda} \E \biggl( \int_{\R} (f'(W) - f'(W+t)) \, \widehat{K}(t) \,
dt \biggr)\\
& + & \frac{1}{\lambda} \E (f(W) R(W)) \bigr|.
\end{eqnarray*}
Using $\int_{\R} |t| \widehat{K}(t) \, dt = \frac 12 \E |W-W'|^3$, the bounds in Assumption (B1) give:
\begin{equation}
|\E h(W) - \E h(Z) | \leq \|h'\| \biggl( c_2 \E \biggl|  1 - \frac{1}{2 \lambda} \E \bigl( (W-W')^2| W \bigr) \biggr| + \frac{c_3}{4 \lambda} \E |W-W'|^3 + \frac{c_1}{\lambda} \sqrt{ \E (R^2)} \biggr).
\end{equation}

{\bf Proof of (2):}
Now let $f=f_z$ be the solution of the Stein equation \eqref{steinid3}. As in \eqref{tis}, using \eqref{start3}, we obtain
\begin{eqnarray*}
P(W \leq z) - P(z) & = & \E (f'(W) + \psi(W) f(W)) \\
& =&  \E \biggl( f'(W) \bigl( 1 - \frac{1}{2 \lambda} (W-W')^2 \bigr) \biggr) + \frac{1}{2 \lambda} \E (2 f(W) \, R)\\
& & - \frac{1}{2 \lambda} \E \biggl[ (W-W') \bigl( f(W) - f(W') - (W-W')f'(W) \bigr) \biggr]\\
& = & T_1 + T_2 + T_3.
\end{eqnarray*}
Now the bounds in Assumption (B2) give
$$
|T_1| \leq d_2 \sqrt{ \E \biggl( 1 - \frac{1}{\lambda}  \E[(W'-W)^2 |W] \biggr)^2}.
$$
and
$$
|T_2| \leq \frac{d_1}{\lambda} \sqrt{ \E(R^2)}.
$$
Using the decomposition \eqref{t3} of $(-2 \lambda) \, T_3$,
the modulus of the first term can be bounded by $d_3 \, \E \bigl( (W-W')^2 1_{\{|W-W'| >A\}} \bigr)$.
Using the Stein identity \eqref{steinid3}, the second summand can be represented
as
\begin{eqnarray*}
& & \E \biggl( (W-W') 1_{\{|W-W'| \leq A\}} \int_{-(W-W')}^0 \bigl( -\psi(W+t) \, f(W+t) + \psi(W) \, f(W) \bigr) dt \biggr) \\
& & + \E \biggl( (W-W') 1_{\{|W-W'| \leq A\}} \int_{-(W-W')}^0 (1_{\{W+t \leq z\}} - 1_{\{W \leq z\}}) dt \biggr) =: U_1 + U_2.
\end{eqnarray*}
With $g(x) := (\psi(x) f(x))'$ we obtain
$$
-\psi(W+t) \, f(W+t) + \psi(W) \, f(W) = - \int_0^t g(W+s) \, ds.
$$
Since $|g(x)| \leq d_4$ we obtain $|U_1| \leq \frac{A^3}{2} d_4$.

Analogously to the steps in the proof of Theorem \ref{ourmain}, $U_2$ can be bounded by
$$
\E \bigl( (W-W')(f(W)-f(W')) \bigr) = 2 \E \bigl( f(W) \, R(W) \bigr) - 2 \lambda \, \E \bigl( \psi(W) \, f(W) \bigr),
$$
where we applied \eqref{start3}, and where
$f$ is defined by $f(x):= -1.5 A$ for $x \leq z-A$, $f(x):=1.5 A$ for $x \geq z+2A$ and $f(x):=x-z-A/2$
in between.
Thus $U_2 \leq 3A \bigl( \E(|R|) + \lambda \E (|\psi(W)|) \bigr)$. Similarly we obtain $U_2 \geq - 3A \bigl( \E(|R|) + \lambda \E (|\psi(W)|) \bigr)$.
\end{proof}

\begin{proof}[Proof of Theorem \ref{generaldensity2}]
The main observation is the following identity:
\begin{eqnarray*}
\E \bigl( -\E[W \psi(W)] \, f'(W) + \psi(W) f(W) \bigr) & = & \E \biggl( f'(W) \biggl( \frac{\E[(W-W')^2] - 2 \E[W \, R]}{2 \lambda} \biggr) \biggr) +
\E \bigl( \psi(W) \, f(W) \bigr) \nonumber \\
& & \hspace{-6cm} = \E \biggl( f'(W) \biggl( \frac{\E[(W-W')^2] - \E[(W-W')^2|W]}{2\lambda} \biggr) \biggr) +\frac{1}{\lambda} \biggl(\E[f(W) \, R] - \E[\E(W R) \, f'(W)] \biggr) +T_3
\end{eqnarray*}
with $T_3$ defined as in the proof of Theorem \ref{generaldensity}. Now we can apply the Cauchy-Schwarz inequality to get
$$
\E \bigl| \E[(W-W')^2] - \E[(W-W')^2|W] \bigr| \leq \bigl( {\rm Var} \bigl( \E[(W-W')^2|W] \bigr) \bigr)^{1/2}.
$$
Now the proof follows the lines of the proof of Theorem \ref{generaldensity}.
%
%STIMMT so allgemein auch nicht, denn nun muss man fuer die neue stein-gleichung und ihre loesungen klaeren,
%dass konstante ebenfalls existieren, was allgemein unkalr ist, aber ich glaueb, ich nehme es einfach an, weil
%es in den anwendungen ohnehin kalr sein wird...
\end{proof}

\begin{remark} \label{alternative}
We discuss an alternative bound in Theorem \ref{generaldensity} in the case that $(\psi(x) f_z(x))'$ cannot be bounded uniformly.
By the mean value theorem we obtain in general
$$
-\psi(W+t) f(W+t) + \psi(W) f(W) = \psi(W) \bigl( - \int_0^1 f'(W+st) t ds \bigr) + f(W+t) \bigl( - \int_0^1 \psi'(W+st) t ds \bigr).
$$
This gives
$$
|-\psi(W+t) f(W+t) + \psi(W) f(W)| \leq d_2 |\psi(W)| |t| + d_1 \int_0^1 |\psi'(W+st)| |t| ds.
$$
Now we get the bound
$$
\frac{1}{2 \lambda} |U_1| \leq \frac{d_2 A^3}{4 \lambda} \E(|\psi(W)|) + \frac{d_1}{2 \lambda} \E(V)
$$
with
$$
V := \biggl( |W-W'| \, 1_{\{|W-W'| \leq A\}} \int_{-(W-W')}^0 \int_0^1 |\psi'(W+st)| \, |t| ds \, dt \biggr).
$$
Let us consider the example $\psi(x) = - x^3/3$. Now
$$
\psi'(W+st)| = |(W+st)^2| = | W^2 + 2 st W + s^2t^2|,
$$
hence
$$
|(W+st)^2| |t| \leq |t| |W^2| + |t^2| 2 |W| |s| + |t^3| |s^2|
$$
and integration over $s$ gives
$$
\int_0^1 | \psi'(W+st)| |t| ds \leq |t| |W^2| + |t^2| |W| + |t^3| /3.
$$
Integration over $t$ leads to
$$
\int_{-(W-W')}^0 ( \int_0^1 | \psi'(W+st)| |t| ds) \leq \frac{|\Delta|^2}{2} |W^2| + \frac{|\Delta|^3}{3} |W| + \frac{|\Delta|^4}{12}.
$$
with $\Delta := (W-W')$. Hence we get
$$
\E(V)  \leq \biggl( \frac{A^3}{2} \E |W^2| + \frac{A^4}{3} \E |W| + \frac{A^5}{12} \biggr).
$$
We will see in Section 5, that this bound is good enough for an alternative proof of Theorem \ref{CWcritical}.
\end{remark}

\section{Berry-Esseen bound at the critical temperature}
\begin{proof}[Proof of Theorem \ref{CWcritical}]
We start with \eqref{beta1}, where $W$ is given by \eqref{w2}.  We will calculate the remainder term $R(W)$ more carefully:
By Taylor expansion and the identities $m_i(X) = m(X) - X_i/n$ and $m(X) = \frac{1}{n^{1/4}} W$ we obtain
$$
\frac{1}{n^{3/4}} \frac 1n \sum_{i=1}^n \tanh ( m_i(X)) = \frac 1n W - \frac{1}{n^{3/2}} \frac{W^3}{3} - \mathcal{O} \bigl(\frac{W}{n^2} \bigr) + \mathcal{O}
\bigl( \frac{W^3}{n^{5/2}} \bigr) + \mathcal{O} \bigl( S(W) \bigr)
$$
with
$$
S(W) = \frac{1}{n^{3/4}} \frac 1n \sum_{i=1}^n m_i(X)^5 = \mathcal{O} \bigl( \frac{W^5}{n^2} \bigr) +  \mathcal{O} \bigl( \frac{W^3}{n^{7/2}} \bigr) +
 \mathcal{O} \bigl( \frac{W^2}{n^{21/4}} \bigr) +  \mathcal{O} \bigl( \frac{W}{n^6} \bigr).
$$
From Lemma \ref{momentsgeneral} we know that for $\varrho$ being the symmetric Bernoulli distribution and $\beta=1$ we get
$\E |W|^6 \leq {\rm const.}$.
Using this we get the exchangeable pair identity \eqref{beta1} with
$R(W) = \mathcal{O} \bigl( \frac{1}{n^2} \bigr)$.
With Lemma \ref{genbound}, we can now apply Theorem \ref{generaldensity}, using
$|W - W'| \leq \frac{1}{n^{3/4}}=:A$.
We obtain $1.5 A \, \E(|\psi(W)|) \leq \rm{const.} \frac{1}{n^{3/4}}$ and  $\frac{d_4 \, A^3}{4 \lambda}  = \frac{d_4}{4} \frac{1}{n^{3/4}}$.
Using $\E |R(W)| \leq \rm{const.} \frac{1}{n^2}$ we get
$$
\bigl( d_1 + \frac 32 A \bigr) \frac{\E|R(W)|}{\lambda} \leq \rm{const.} \frac{1}{\sqrt{n}}.
$$
Moreover we obtain
$$
\E \bigl[ (W-W')^2 | \mathcal{F} \bigr] = \frac{2}{n^{3/2}} - \frac{2}{n^{5/2}} \sum_{i=1}^n X_i \tanh (m_i(X)).
$$
Hence applying Theorem \ref{generaldensity} we have to bound the expectation of
$$
T := \bigl| \frac 1n \sum_{i=1}^n X_i \tanh (m_i(X)) \bigr|.
$$
Again using Taylor and $m_i(X) = m(X) - \frac{X_i}{n}$ and Lemma \ref{momentsgeneral}, the leading term of $T$ is $\frac{W^2}{n^{1/2}}$.
Hence $\E(T) = \mathcal{O}(n^{-1/2})$ and Theorem \ref{CWcritical} is proved.
\end{proof}

\begin{remark} In Remark \ref{alternative}, we presented an alternative bound via Stein's method without proving
a uniform bound for $(\psi'(x) f_z(x))'$. As we can see, the additional terms in this bound are of smaller order than $\mathcal{O}(n^{-1/2})$,
using $A = n^{-3/4}$.
\end{remark}

\begin{proof}[Proof of Theorem \ref{CWbetadep}]
(1) Let $\beta_n -1 = \frac{\gamma}{\sqrt{n}}$ and $W= S_n/n^{3/4}$.
For the distribution function $F_{\gamma}$ in Theorem \ref{CWbetadep} we obtain $\psi(x) = \gamma \, x - \frac 13 x^3$. Moreover
we have
\begin{equation} \label{sehrhuebsch}
\E[W-W'|W] = \frac{1 - \beta_n}{n} W + \frac{\beta_n^3}{n^{3/2}} \frac{W^3}{3} + R(\beta_n, W)
\end{equation}
with $R(\beta_n, W) = \mathcal{O}(n^{-2})$. With $\beta_n-1 = \frac{\gamma}{\sqrt{n}}$ we obtain
$$
\E[W-W'|W] = -\frac{\gamma}{n^{3/2}} W + \frac{\beta_n^3}{n^{3/2}} \frac{W^3}{3} + R(\beta_n, W) = - \frac{1}{n^{3/2}} \psi(W) + \tilde{R}(\beta_n, W)
$$
with $\tilde{R}(\beta_n, W) = \mathcal{O}(n^{-2})$. Now we only have to adapt the proof of Theorem \ref{CWcritical} step
by step, using, that the sixth moment of $W$ is bounded for varying $\beta_n$, see Lemma \ref{momentsgeneral}. Hence by
Lemma \ref{genbound} and Theorem \ref{generaldensity}, part (1) is proved.

(2): we consider the case $|\beta_n -1| = \mathcal{O}(n^{-1})$ and $W= S_n/ n^{3/4}$.
Now in \eqref{sehrhuebsch}, the term $ \frac{1 - \beta_n}{n} W$ will be a part of the remainder:
$$
\E[W-W'|W] = \frac{\beta_n^3}{n^{3/2}} \frac{W^3}{3} + R(\beta_n, W) + \frac{1 - \beta_n}{n} W =: -\frac{\beta_n^3}{n^{3/2}} \psi(W) + \hat{R}(\beta_n, W)
$$
with $\psi(x) := x^3/3$.
Along the lines of the proof of Theorem \ref{CWcritical}, we have to bound $\frac{\E|\hat{R}|}{\lambda}$ with $\lambda = \frac{\beta_n^3}{n^{3/2}}$.
But since by assumption
$$
\lim_{n \to \infty} \frac{1}{\lambda} \frac{(1-\beta_n)}{n} = \frac{\sqrt{n}(1-\beta_n)}{\beta_n^3} = 0,
$$
applying Theorem \ref{generaldensity}, we obtain the convergence in distribution for any $\beta_n$ with  $|\beta_n -1| \ll n^{-1/2}$, and we obtain the Berry-Esseen bound
of order $\mathcal{O}(1 / \sqrt{n})$ for any $|\beta_n -1| = \mathcal{O} (n^{-1})$.

(3) Finally we consider $|\beta_n -1| \gg n^{-1/2}$ and $W= \sqrt{\frac{(1-\beta_n)}{n}} S_n$.
Now we obtain
$$
\E[W-W'|W] = \frac{1 - \beta_n}{n} W + \frac{\beta_n \, W}{n^2} + \frac{\beta_n^3}{n^2(1-\beta_n)} \frac{W^3}{3}
+ R(\beta_n, W) =: -\lambda \psi(W) + \tilde{R}(\beta_n,W)
$$
with $\lambda = \frac{(1-\beta_n)}{n}$ and $\psi(x) = -x$. We apply Corollary \ref{corsigma}: with $A= \frac{1}{\sqrt{n}}(1-\beta_n)^{1/2}$,
one obtains $\lambda^{-1} A^3 = n^{-1/2} (1-\beta_n)^{1/2}$ and
$$
\frac{\E|\tilde{R}(\beta_n,W)|}{\lambda} \leq \frac{{\rm const}}{n(1-\beta_n)^2}.
$$
Moreover
$$
\E[(W-W')^2|W] = \frac{2(1-\beta_n)}{n} - \frac{2(1-\beta_n)}{n} \frac 1n \sum_{i=1}^n X_i \tanh \bigl( \beta_n m_i(X) \bigr).
$$
Hence
$$
\biggl| 1 - \frac{1}{2 \lambda} \E[(W-W')^2|W] \biggr| = \biggl| \frac{\beta_n}{n (1- \beta_n)} W^2 - \frac{\beta_n}{n} - \frac{\beta_n^3}{n^2(1-\beta_n)^2}
\frac{W^4}{3} + R(\beta_n, W) \biggr| = \mathcal{O} \bigl( \frac{\beta_n}{n(1-\beta_n)} \bigr).
$$
Hence with $|\beta_n -1| \gg n^{-1/2}$ we obtain convergence in distribution. Under the additional assumption
$|\beta_n -1| \gg n^{-1/4}$ we obtain the Berry-Esseen result.
\end{proof}

\section{Proof of the general case}
\begin{proof}[Proof of Theorem \ref{CWgeneral}]
Given $\varrho$ which satisfies the GHS-inequality and let $\alpha$ be the global minimum of type $k$ and strength $\mu(\alpha)$
of $G_{\varrho}$. In case $k=1$ it is known that the random variable $\frac{S_n}{\sqrt{n}}$ converges in distribution to
a normal distribution $N(0, \sigma^2)$ with $\sigma^2 = \mu(\alpha)^{-1} - \beta^{-1}= (\sigma_{\varrho}^{-2} - \beta)^{-1}$, see
for example \cite[V.13.15]{Ellis:LargeDeviations}.
Hence in this case we will apply Corollary \ref{corsigma2} (to obtain better constants for our Berry-Esseen bound in comparison to Theorem
\ref{generaldensity2}).

Consider $k \geq 1$. We just treat the case $\alpha=0$ and denote $\mu = \mu(0)$. The more general case
can be done analogously.
For $k=1$, we consider $\psi(x) = - \frac{x}{\sigma^2}$ with $\sigma^2 = \mu^{-1} - \beta^{-1}$. For any $k \geq 2$
we consider
$$
\psi(x) = - \frac{\mu}{(2k-1)!} x^{2k-1}.
$$
We define
$$
W := W_{k,n} := \frac{1}{n^{1 - 1/(2k)}} \sum_{i=1}^n X_i
$$
and $W'$, constructed as in Section 3, such that
$$
W-W' = \frac{X_I - X_I'}{n^{1 - 1/(2k)}}.
$$
We obtain
$$
\E[W-W'| \mathcal{F}] = \frac 1n W - \frac{1}{n^{1 - 1/(2k)}} \frac 1n \sum_{i=1}^n \E(X_i' | \mathcal{F}).
$$
Now we have to calculate the conditional distribution at site $i$ in the general case:

\begin{lemma} \label{CWidentitygeneral}
In the situation of Theorem \ref{CWgeneral}, if $X_1$ is $\varrho$-a.s. bounded, we obtain
$$
\E (X_i' | \mathcal{F}) = \bigl( m_i(X) - \frac{1}{\beta} G_{\varrho}'(\beta, m_i(X)) \bigr) \, \bigl(1 + \mathcal{O}(1/n) \bigr)
$$
with $m_i(X) := \frac 1n \sum_{j \not= i} X_j = m(X) - \frac{X_i}{n}$.
\end{lemma}

\begin{proof}
We compute the conditional density $g_{\beta}(x_1|(X_i)_{i \ge 2})$ of $X_1=x_1$ given $(X_i)_{i \ge 2}$ under the Curie-Weiss measure:
\begin{eqnarray*}
g_{\beta}(x_1|(X_i)_{i \ge 2})& = & \frac {e^{\beta/2n (\sum_{i \ge 2} x_1 X_i+ \sum_{i \neq j \ge 2} X_i X_j+ x_1^2)}}
{\int e^{\beta/2n (\sum_{i \ge 2} x_1 X_i+ \sum_{i \neq j \ge 2} X_i X_j+ x_1^2)} \varrho(dx_1)}\\
&=& \frac {e^{\beta/2n (\sum_{i \ge 2} x_1 X_i+ x_1^2)}}
{\int e^{\beta/2n (\sum_{i \ge 2} x_1 X_i+ x_1^2)} \varrho(dx_1)}.
\end{eqnarray*}
Hence we can compute $\mathbb{E}[X_1'|\mathcal{F}]$ as
$$
\mathbb{E}[X_1'|\mathcal{F}]=\frac {\int x_1 e^{\beta/2n (\sum_{i \ge 2} x_1 X_i+ x_1^2)}\varrho(dx_1)}
{\int e^{\beta/2n (\sum_{i \ge 2} x_1 X_i+ x_1^2)} \varrho(dx_1)}.
$$
Now, if $|X_1| \le c$ $\varrho$-a.s
$$
\mathbb{E}[X_1'|\mathcal{F}]\le \frac {\int x_1 e^{\beta/2n (\sum_{i \ge 2} x_1 X_i)}\varrho(dx_1) e^{\beta c^2/2n}}
{\int e^{\beta/2n (\sum_{i \ge 2} x_1 X_i+ x_1^2)} \varrho(dx_1)e^{-\beta c^2/2n}}
$$
and
$$
\mathbb{E}[X_1'|\mathcal{F}]\ge \frac {\int x_1 e^{\beta/2n (\sum_{i \ge 2} x_1 X_i)}\varrho(dx_1) e^{-\beta c^2/2n}}
{\int e^{\beta/2n (\sum_{i \ge 2} x_1 X_i+ x_1^2)} \varrho(dx_1)e^{\beta c^2/2n}}.
$$
By computation of the derivative of $G_\varrho$ we see that
$$
\frac {\int x_1 e^{\beta/2n (\sum_{i \ge 2} x_1 X_i)}\varrho(dx_1)}
{\int e^{\beta/2n (\sum_{i \ge 2} x_1 X_i+ x_1^2)} \varrho(dx_1)} e^{\pm \beta c^2/n} = \bigl( m_1(X) - \frac{1}{\beta} G_{\varrho}'(\beta, m_1(X))
\bigr) \, (1 \pm \beta c^2/n).
$$
\end{proof}

\begin{remark} If we consider the Curie-Weiss model with respect to $\widehat{P}_{n, \beta}$, the conditional density $g_{\beta}(x_1 |(X_i)_{i \geq 2})$
under this measure becomes
$$
g_{\beta}(x_1|(X_i)_{i \ge 2}) = \frac {e^{\beta/2n (\sum_{i \ge 2} x_1 X_i)}}
{\int e^{\beta/2n (\sum_{i \ge 2} x_1 X_i)} \varrho(dx_1)}.
$$
Thus we obtain $\E (X_i' | \mathcal{F}) = \bigl( m_i(X) - \frac{1}{\beta} G_{\varrho}'(\beta, m_i(X)) \bigr)$ without
the boundedness assumption for the $X_1$.
\end{remark}

Applying Lemma \ref{CWidentitygeneral} and the presentation \eqref{Taylor}
of $G_{\varrho}$, it follows that
$$
\E[W-W'|W] = \frac 1n W - \frac{1}{n^{1-1/(2k)}} \biggl(
\frac 1n \sum_{i=1}^n \biggl( m_i(X) - \frac{\mu}{\beta (2k-1)!} m_i(X)^{2k-1} + \mathcal{O} \bigl( m_i(X)^{2k} \bigr) \biggr) \biggr).
$$
With $m_i(X) = m(X) - \frac{X_i}{n}$ and $m(X) = \frac{1}{n^{1/(2k)}}W$ we obtain
$$
\frac{1}{n^{1-1/(2k)}} \frac 1n \sum_{i=1}^n m_i(X) = \frac 1n W - \frac{1}{n^2} W
$$
and
\begin{eqnarray*}
\frac{1}{n^{1-1/(2k)}} \frac 1n \sum_{i=1}^n \frac{\mu}{\beta (2k-1)!} m_i(X)^{2k-1} =\frac{1}{n^{1-1/(2k)}}
\frac{\mu}{\beta (2k-1)!} \sum_{l=0}^{2k-1} {2k -1 \choose l} m(X)^{2k-1-l} \frac{(-1)^{l}}{n^l} \frac 1n \sum_{i=1}^n X_i^l.
\end{eqnarray*}
For any $k \geq 1$ the first summand ($l=0$) is
\begin{equation} \label{schick}
\frac{1}{n^{2- \frac 1k}}  \frac{\mu}{\beta (2k-1)!} W^{2k-1} = - \frac{1}{n^{2- \frac 1k}} \psi(W).
\end{equation}
To see this, let $k=1$. Since we set $\phi''(0)=1$, we obtain $\mu(0)= \beta - \beta^2$ and therefore
$\frac{1}{\beta} \mu(0) W = (1-\beta) W$. In the case $k \geq 2$ we know that $\beta=1$. Hence in both cases,
\eqref{schick} is checked. Summarizing we obtain for any $k \geq 1$
$$
\E[W-W'|W] =  - \frac{1}{n^{2- \frac 1k}} \psi(W) + R(W) =: - \lambda \psi(W) + R(W)
$$
with
\begin{eqnarray*}
R(W) & = & \frac{1}{n^2} W + \frac{\mu}{\beta (2k-1)!} \frac{1}{n^{1-1/(2k)}}
\sum_{l=1}^{2k-1} {2k -1 \choose l} m(X)^{2k-1-l} \frac{(-1)^{l}}{n^l} \frac 1n \sum_{i=1}^n X_i^l + \mathcal{O}(m(X)^{2k}) \\
& = & \frac{1}{n^2} W + \frac{\mu}{\beta (2k-1)!}
\sum_{l=1}^{2k-1} {2k -1 \choose l} \frac{1}{n^{2 - \frac 1k - \frac{l}{2k}}} W^{2k-1-l}  \frac{(-1)^{l}}{n^l} \frac 1n \sum_{i=1}^n X_i^l + \mathcal{O}(\frac{W^{2k}}{n}).
\end{eqnarray*}
With Lemma \ref{momentsgeneral} we know that $\E |W|^{2k} \leq \rm{const}$.
We will apply Corollary \ref{corsigma2}, if $k=1$ and Theorem \ref{generaldensity2} for $k \geq 2$. In both cases we apply
Lemma \ref{genbound}. Since the spin variables are assumed to be bounded $\varrho$-a.s,
we have
$$
|W-W'| \leq \frac{\rm{const.}}{n^{1 - \frac{1}{2k}}} =:A.
$$

Let $k=1$. Now $\lambda= \frac 1n$, $A= {\rm const.}/n^{-1/2}$, $\E(W^4) \leq {\rm const}$. The leading term
of $R$ is $W/n^2$. Hence the last four summands in \eqref{mainbound3} of Corollary \ref{corsigma2} are $\mathcal{O} (n^{-1/2})$.

For $k \geq 2$ we obtain
$\frac{3 A}{2} \E(|\psi(W)|) = \mathcal{O} \bigl( n^{\frac{1}{2k}-1} \bigr)$ and
$\frac{1}{\lambda} \bigl( \frac{d_4 A^3}{4} \bigr) =  \mathcal{O} \bigl( n^{\frac{1}{2k}-1} \bigr)$.
The leading term in the second term of $R(W)$ is the first summand ($l=1$), which is of order
$\mathcal{O}(n^{-3+ \frac{1}{k} + \frac{1}{2k}})$.
With $\lambda= n^{\frac{1}{k} -2}$ we obtain
$$
\frac{\E(|R|)}{\lambda} \leq \frac{\E(|W|)}{\lambda n^2} +  \mathcal{O} \bigl( n^{\frac{1}{2k} -1} \bigr) \quad \text{and}
\quad \frac{\E(|W|)}{\lambda n^2} = \mathcal{O} \bigl( n^{1/k} \bigr).
$$
Hence the last four summands in \eqref{kolall2} of Theorem \ref{generaldensity2} are $\mathcal{O} (n^{-1/k})$.

Finally we have to consider the variance of $\frac{1}{2 \lambda}\E[(W-W')^2|W]$. Hence we have to bound the variance of
\begin{equation} \label{haupt1}
\frac{1}{2n} \sum_{i=1}^m X_i^2 + \frac{1}{2n} \sum_{i=1}^n E[ (X_i')^2|\mathcal{F}] + \frac 1n \sum_{i=1}^n X_i \biggl( m_i(X) - \frac{1}{\beta}
G_{\varrho}'(\beta, m_i(X)) \biggr) \,  (1 + O(1/n)).
\end{equation}
Since we assume that $\varrho \in {\rm GHS}$, we can apply the correlation-inequality due to Lebowitz (see Remark \ref{lebo})
$$
\E(X_i X_j X_k X_l) - \E(X_i X_j) \E(X_k X_l)  - \E(X_i X_k) \E(X_j X_l)  - \E(X_i X_l) \E(X_j X_k) \leq 0.
$$
The choice $i=k$ and $j=l$ leads to the bound
$$
{\rm Cov} \bigl( X_i^2, X_j^2) = \E(X_i^2 X_j^2) - \E(X_i^2) \E(X_j^2) \leq 2 (\E (X_i \, X_j))^2.
$$
With Lemma \ref{momentsgeneral} we know that $(\E(X_i X_j))^2 \leq {\rm const.} n^{-2/k}$. This gives
$$
{\rm Var} \bigl( \frac{1}{2n}\sum_{i=1}^n X_i^2 \bigr) = \frac{1}{4 n^2} \sum_{i=1}^n {\rm Var}(X_i^2) + \frac{1}{4 n^2} \sum_{1 \leq i < j \leq n} {\rm Cov} (X_i^2, X_j^2)
= \mathcal{O} \bigl( n^{-1} \bigr) + \mathcal{O} \bigl(n^{-2/k}).
$$
Using a conditional version of Jensen's inequality we have
$$
{\rm Var} \bigl( \E \bigl( \frac{1}{2n}\sum_{i=1}^n X_i^2 \bigl| \mathcal{F} \bigr) \bigr) \leq {\rm Var} \bigl( \frac{1}{2n}\sum_{i=1}^n X_i^2 \bigr).
$$
Hence the variance of the second term in \eqref{haupt1} is of the same order as the variance of the first term.
Applying \eqref{Taylor} for $G_{\varrho}$, the variance of the third term in \eqref{haupt1} is of the order of the variance of $W^2 / n^{1/k}$.
Summarizing the variance of \eqref{haupt1} can be bounded by 9 times the maximum of the variances of the three terms in \eqref{haupt1}, which
is a constant times $n^{-2/k}$, and therefore for $k \geq 1$ we obtain
$$
\biggl( {\rm Var} \biggl(  \frac{1}{2 \lambda} \, \E[(W-W')^2|W] \biggr) \biggr)^{1/2} = \mathcal{O}(n^{-1/k}).
$$
Note that for $k \geq 2$
$$
\frac{\psi(x)}{-\E[W \psi(W)]} = -\frac{x^{2k-1}}{\E(W^{2k})}.
$$
Hence we compare the distribution of $W$ with a distribution with Lebesgue-probability density proportional to
$\exp \bigl( -\frac{x^{2k}}{2k \E(W^{2k})} \bigr)$.
\end{proof}

\begin{proof}[Proof of Theorem \ref{CWgendep}]
Since $\alpha=0$ and $k=1$ for $\beta \not=1$ while $\alpha=0$ and $k \geq 2$ for $\beta =1$, $G_{\varrho}(\cdot)$ can now be expanded as
$$
G(s) = G(0) + \frac{\mu_1}{2} s^2 + \frac{\mu_k}{(2k)!} s^{2k} + \mathcal{O} (s^{2k+1}) \quad \text{as} \quad s \to 0.
$$
Hence $\frac{1}{\beta_n} \, G_{\varrho}'(s) = \frac{\mu_1}{\beta_n} s + \frac{\mu_k}{\beta_n (2k-1)!} s^{2k-1} + \mathcal{O}(s^{2k})$.
With Lemma \ref{CWidentitygeneral} and $\mu_1=(1-\beta_n) \beta_n$ we obtain
$$
\E[X_i | \mathcal{F}] = \beta_n  m_i(X) - \frac{\mu_k}{\beta_n (2k-1)!} m_i(X)^{2k-1} \, (1 + \mathcal{O}(1/n)).
$$
We get
$$
\E[W-W'|W] = \frac{1 -\beta_n}{n} W + \frac{\beta_n}{n^2} W + \frac{1}{n^{2 - 1/k}} \frac{\mu_k}{\beta_n (2k-1)!} W^{2k+1} + R(\beta_n, W).
$$
The remainder $R(\beta_n,W)$ is the remainder in the proof of Theorem \ref{CWgeneral} with $\mu$ exchanged by $\mu_{k}$ and $\beta$
exchanged by $\beta_n$.

\noindent
Let $\beta_n -1 = \frac{\gamma}{n^{1-1/k}}$ and $W= n^{1/(2k)-1} \sum_{i=1}^n X_i$. We obtain
\begin{equation} \label{wunderbar}
\E[W-W'|W] = - \frac{1}{n^{2 - 1/k}} \psi(W) + \frac{\beta_n}{n^2} W + R(\beta_n,W),
\end{equation}
where $\psi(x) = \gamma x - \frac{\mu_k}{\beta_n \, (2k-1)!} x^{2k-1}$. As in the proof of Theorem \ref{CWgeneral} we
obtain that $R(\beta_n,W)= \mathcal{O}(n^{-2})$. Now we only have to adapt the proof of Theorem \ref{CWgeneral} step by step,
applying Lemma \ref{momentsgeneral}, Lemma \ref{genbound} and Theorem \ref{generaldensity2}.

\noindent
Let $|\beta_n -1| = \mathcal{O}(1/n)$ and $W= n^{1/(2k)-1} \sum_{i=1}^n X_i$. Now in \eqref{wunderbar}, the term $\frac{1-\beta_n}{n} W$ will be a part of the remainder:
\begin{eqnarray*} %\label{wunderbar}
\E[W-W'|W] & = & \frac{1}{n^{2 - 1/k}} \frac{\mu_k}{\beta_n (2k-1)!} W^{2k+1} + R(\beta_n, W) +  \frac{\beta_n}{n^2} W + \frac{1 -\beta_n}{n} W \\
& =: & - \frac{1}{\beta_n \, n^{2 - 1/k}} \psi(W) + \hat{R}(\beta, W)
\end{eqnarray*}
with $\psi(x) = - \frac{\mu_k}{(2k-1)!} x^{2k-1}$.
Following the lines of the proof of Theorem \ref{CWgeneral}, we have to bound $\frac{\E|\hat{R}(\beta_n,W)|}{\lambda}$ with $\lambda:= \frac{1}{\beta_n \, n^{2-1/k}}$.
Since by our assumption for $(\beta_n)_n$ we have
$$
\lim_{n \to \infty} \frac{1}{\lambda} \frac{(1-\beta_n)}{n} = \beta_n (1- \beta_n) n^{1-1/k} =0.
$$
Thus with Theorem \ref{generaldensity2} we obtain convergence in distribution for any $\beta_n$ with $|\beta_n -1| \ll n^{-(1-1/k)}$.
Moreover we obtain the Berry-Esseen bound of order $\mathcal{O}(n^{-1/k})$ for any $|\beta_n-1| = \mathcal{O}(n^{-1})$.

\noindent
Finally we consider $|\beta_n -1| \gg n^{-(1-1/2)}$ and $W= \sqrt{\frac{(1-\beta_n)}{n}} S_n$. A little calculation gives
$$
\E[W-W'|W] = \frac{1-\beta_n}{n} W + \frac{\beta_n \, W}{n^2} + \frac{\mu_k}{(2k-1)! n^k (1-\beta_n)^{k-1} \beta_n} W^{2k-1} + R(\beta_n,W) =: -\lambda \psi(W) +
\hat{R}(\beta_n, W)
$$
with $\psi(x) = -x$ and $\lambda = \frac{1-\beta_n}{n}$. Now we apply Corollary \ref{corsigma2}. With $A := \frac{\rm{const.} (1-\beta_n)^{1/2}}{\sqrt{n}}$
we obtain
$$
\frac{A^3}{\lambda} \leq \frac{\rm{const.} (1-\beta_n)^{1/2}}{\sqrt{n}} \quad \text{and} \quad \frac{\E| \hat{R}(\beta_n,W)|}{\lambda} \leq
\frac{\rm{const}}{n^{k-1}(1-\beta_n)^k}.
$$
Remark that the bound on the right hand side is good for any $|\beta_n -1| \gg n^{-(1-1/k)}$.
Finally we have to bound the variance of $\frac{1}{2 \lambda} \, \E[(W-W')^2|W]$. The leading term is the variance
of
$$
\frac 1n \sum_{i=1}^n X_i \biggl( m_i(X) - \frac{1}{\beta}
G_{\varrho}'(\beta, m_i(X)) \biggr),
$$
which is of order $\mathcal{O} \bigl( \frac{\beta_n}{n(1-\beta_n)} \bigr)$. Hence with
$|\beta_n -1| \gg n^{-(1-1/k)}$ we get convergence in distribution.
Under the additional assumption that $|\beta_n -1| \gg n^{-(1/2-1/(2k))}$ we obtain the Berry-Esseen bound.
\end{proof}

\begin{proof}[Proof of Theorem \ref{wasserstein}]
We apply Theorem \ref{generaldensity2}. For unbounded spin variables $X_i$ we consider $\widehat{P}_{n, \beta}$ and apply Lemma \ref{CWidentitygeneral}
to bound $\frac{1}{\lambda} \sqrt{ {\rm Var} ( \E[(W-W')^2|W] )}$ exactly as in the proof of Theorem \ref{CWgeneral}.
By Theorem \ref{generaldensity2} it remains to bound $\frac{1}{\lambda} \E|W-W'|^3$. With $\lambda = n^{-2+1/k}$ we have
$$
\frac{1}{\lambda} \E|W-W'|^3 = \frac{1}{n^{1- 1/2k}} \E|X_I - X_I'|^3 = \frac{1}{n^{1- 1/2k}} \E|X_1 - X_1'|^3.
$$
Now $ \E|X_1 - X_1'|^3 \leq \E|X_1|^3 + 3 \E|X_1^2 \, X_1'| + 3 \E |X_1 (X_1')^2| + \E |X_1'|^3$.
Using H\"older's inequality we obtain
$$
\E|X_1^2 \, X_1'| \leq \bigl( \E|X_1|^3 \bigr)^{2/3} \, \bigl( \E|X_1'|^3 \bigr)^{1/3} \leq \max \bigl( \E |X_1|^3, \E |X_1'|^3 \bigr).
$$
Hence we have
$$
\frac{1}{\lambda} \E|W-W'|^3 \leq \frac{8}{n^{1 - 1/2k}} \max \bigl( \E |X_1|^3, \E |X_1'|^3 \bigr).
$$
Thus the Theorem is proved.
\end{proof}

\section{Examples}
It is known that the following distributions $\varrho$ are ${\rm GHS}$ (see \cite[Theorem 1.2]{Ellis/Monroe/Newman:1976}).
The symmetric Bernoulli measure is ${\rm GHS}$, first noted in \cite{Ellis:1975}. The family of measures
$$
\varrho_a(dx) = a \, \delta_x + \bigl( (1-a)/2 \bigr) \bigl( \delta_{x-1} + \delta_{x+1} \bigr)
$$
for $0 \leq a \leq 2/3$ is ${\rm GHS}$, whereas the GHS-inequality fails for $2/3 < a < 1$, see \cite[p.153]{Griffiths/Simon:1973}.
${\rm GHS}$ contains all measures of the form
$$
\varrho_V(dx) := \bigl( \int_{\R} \exp \bigl( -V(x) \bigr) \, dx \bigr)^{-1} \, \exp \bigl( -V(x) \bigr)\, dx,
$$
where $V$ is even, continuously differentiable, and unbounded above at infinity, and $V'$ is convex on $[0, \infty)$.
${\rm GHS}$ contains all absolutely continuous measures $\varrho \in {\mathcal B}$ with support on $[-a,a]$
for some $0 < a < \infty$ provided $g(x) = d\varrho/dx$ is continuously differentiable and strictly
positive on $(-a,a)$ and $g'(x)/g(x)$ is concave on $[0,a)$. Measures like $\varrho(dx) = {\rm const.} \exp \bigl( -a x^4 - b x^2 \bigr) \, dx$
or $\varrho(dx) = {\rm const.} \exp \bigl( - a \cosh x - b x^2 \bigr) \, dx$ with $a >0$ and $b$ real are GHS. Both are of physical interest, see
\cite{Ellis/Monroe/Newman:1976} and references therein).

\noindent
\begin{example}[A Curie--Weiss model with three states]
We will now consider the next simplest example of the classical Curie--Weiss model: a model with three states.
Observe, that this is not the Curie--Weiss--Potts model \cite{Ellis/Wang:1990}, since the latter has a different Hamiltonian.
Indeed the Hamiltonian considered in \cite{Ellis/Wang:1990} is of the form $\frac 1n \sum_{i,j} \delta_{x_i, x_j}$. It favours
states with many equal spins, whereas in our case the spins also need to have large values.
We choose $\varrho$ to be
$$
\varrho=\frac 23 \delta_0 + \frac 16 \delta_{-\sqrt 3}+ \frac 16 \delta_{\sqrt 3}.
$$
This model seems to be of physical relevance. It is studied in \cite{Thompson:book}. In \cite{Blume/Emery/Griffiths:1971}
it was used to analyze the tri-critical point of liquid helium.
A little computation shows that
$$
\frac {d^3}{ds^3} \phi_\varrho(s)=  - 6\,{\displaystyle \frac {{\rm sinh}(x\,\sqrt{3})\,\sqrt{3}\,(
{\rm cosh}(x\,\sqrt{3}) - 1)}{12\,{\rm cosh}(x\,\sqrt{3}) + 6\,
{\rm cosh}(x\,\sqrt{3})^{2} + {\rm cosh}(x\,\sqrt{3})^{3} + 8}} \le 0
$$
for all $s \ge 0$. Hence the GHS-inequality \eqref{GHS} is fulfilled (see also \cite[Theorem 1.2]{Ellis/Monroe/Newman:1976}),
which implies that there is one critical temperature
$\beta_c$ such that there is one minimum of $G$ for $\beta\le \b_c$ and two minima above $\beta_c$.
Since ${\rm Var}_\varrho(X_1)=
2 \frac 1 6 \cdot 3 =1$  we see that $\beta_c=1$. For $\beta\le \b_c$ the minimum of $G$ is located in zero while for
$\beta> 1$ the two minima are symmetric and satisfy
$$
s=
\frac{\sqrt{3}\sinh(\sqrt{3}\beta s)}{2 + \cosh(\sqrt{3}\beta s\,)}.
$$
Now Theorem \ref{CWgeneral} and \ref{CWgendep} tell that
\begin{itemize}
\item For $\beta<1$ the rescaled magnetization $S_n/ \sqrt{n}$ satisfies a Central Limit Theorem and the
limiting variance is $(1-\beta)^{-1}$. Indeed, $\frac {d^2}{ds^2} \phi_\varrho(0)={\rm Var}_\varrho(X_1)=1$.
Hence $\mu_1= \beta-\beta^2$ and $\sigma^2 =\frac 1 {1-\beta}$. Moreover we obtain
$$
\sup_{z \in \R} \bigg| P_n \bigl( \frac{S_n}{\sqrt{n}} \leq z \bigr) - \Phi_W(z) \bigg| \leq \frac{C}{\sqrt{n}}.
$$
\item For $\beta=\b_c=1$ the rescaled magnetization $S_n/ n^{5/6}$ converges in distribution
to $X$ which has the density $f_{3, 6, 1}$. Indeed $\mu_2$ is computed to be 6. Moreover we obtain
$$
\sup_{z \in \R} \bigg| P_n \bigl( \frac{S_n}{n^{5/6}} \leq z \bigr) - \widehat{F}_3(z) \bigg| \leq \frac{C}{n^{1/3}}
$$
where the derivative of $\widehat{F}_3$ is the rescaled density $\exp \bigl( - \frac{x^6}{6 \E(W^6)} \bigr)$.
\item If $\beta_n$ converges monotonically to $1$ faster than $n^{-2/3}$ then $\frac{S_n}{n^{5/6}}$ converges in distribution
to $\widehat{F}_3$, whereas
if $\beta_n$ converges monotonically to $1$ slower than $n^{-2/3}$ then $\frac{\sqrt{1-\beta_n}\, S_n}{\sqrt{n}}$ satisfies
a Central Limit Theorem. Eventually, if
$|1-\beta_n|= \gamma n^{-2/3}$, $\frac{S_n}{n^{5/6}}$ converges in distribution to a random variable which probability distribution
has the mixed Lebesgue-density
$$
\exp \biggl( - c_W^{-1} \biggl( \frac{x^6}{120} - \gamma \frac{x^2}{2} \biggr) \biggr)
$$
with $c_W = \frac{1}{120} \E(W^6) - \gamma \E(W^2)$.
Moreover we have
$$
\sup_{z \in \R} \bigg| P_n \biggl(\frac{S_n}{n^{5/6}} \leq z \biggr) - \frac{1}{Z} \int_{- \infty}^z \exp \biggl( - c_W^{-1} \biggl( \frac{x^6}{120} - \gamma \frac{x^2}{2} \biggr) \biggr) \bigg| \leq \frac{C}{n^{1/3}}.
$$
\end{itemize}
\end{example}

\noindent
\begin{example}[A continuous Curie--Weiss model]
Last but not least we will treat an example of a continuous Curie--Weiss model. We choose as underlying distribution
the uniform distribution on an interval in $\R$. To keep the critical temperature one we define
$$
\frac{d\varrho (x_i)}{d\, x_i}=  \frac{1}{2a} \mathbb{I}_{[-a,a]}(x_i)
$$
with $a= \sqrt 3$. Then from a general result in \cite[Theorem 2.4]{Ellis/Newman:1978b} (see also \cite[Theorem 1.2]{Ellis/Monroe/Newman:1976})
it follows that $\varrho (x_i)$ obeys the
GHS-inequality \eqref{GHS}. Therefore there exists a critical temperature $\beta_c$,
such that for $\b<\b_c$ zero is the unique global minimum of
$G$ and is of type 1, while at $\beta_c$ this minimum is of type $k \ge 2$.
This $\b_c$ is easily computed to be one. Indeed,
$\mu_1 = \b - \b^2 \phi''(0)=  \b - \b^2 \mathbb{E}_\varrho(X_1^2)= \b(1-\beta)$,
since $\varrho$ is centered and has variance one.
Thus $\mu_1$ vanishes at $\b=\b_c=1$. Eventually for $\b> 1$ there are again two minima which are solutions of
$$
\frac{\sqrt 3 \b}{\tanh(\sqrt 3 \b x)} = \b x + \frac 1x.
$$
Now again by Theorems \ref{CWgeneral} and \ref{CWgendep}
\begin{itemize}
\item For $\beta<1$ the rescaled magnetization $S_n/\sqrt{n}$ obeys a Central Limit Theorem and the limiting variance is
$(1-\beta)^{-1}$. Indeed, since $\mathbb{E}_\varrho(X_1^2)=1$, $\mu_1= \beta-\beta^2$ and $\sigma^2 =\frac 1 {1-\beta}$.
\item For $\beta=\b_c=1$ the rescaled magnetization $S_n/n^{7/8}$ converges in distribution
to $X$ which has the density $f_{4, 6/5, 1}$. Indeed $\mu_2$ is computed to be
$$
-\E_\varrho(X_1^4)+3\mathbb{E}_\varrho(X_1^2)= -\frac 9 5 +3 = \frac 65.
$$
Moreover we obtain
$$
\sup_{z \in \R} \bigg| P_n \biggl( \frac{S_n}{n^{7/8}} \leq z \biggr) - \widehat{F}_4(z) \bigg| \leq \frac{C}{n^{1/4}}
$$
where the derivative of $\widehat{F}_4$ is the rescaled density $\exp \bigl( - \frac{x^8}{8 \E(W^8)} \bigr)$.
\item If $\beta_n$ converges monotonically to $1$ faster than $n^{-3/4}$ then $\frac{S_n}{n^{7/8}}$ converges in distribution
to $\widehat{F}_4$, whereas
if $\beta_n$ converges monotonically to $1$ slower than $n^{-3/4}$ then $\frac{\sqrt{1-\beta_n}\, S_n}{\sqrt{n}}$ satisfies
a Central Limit Theorem. Eventually, if
$|1-\beta_n|= \gamma n^{-3/4}$, $\frac{S_n}{n^{7/8}}$ converges in distribution
to the mixed density
$$
\exp \biggl( - c_W^{-1} \biggl( \frac{6}{5} \frac{x^8}{8!} - \gamma \frac{x^2}{2} \biggr) \biggr)
$$
with $c_W = \frac{5}{6 (8!)} \E(W^8) - \gamma \E(W^2)$.
Moreover we have
$$
\sup_{z \in \R} \bigg| P_n \biggl(\frac{S_n}{n^{7/8}} \leq z \biggr) - \frac{1}{Z} \int_{- \infty}^z \exp \biggl( - c_W^{-1}
\biggl( \frac{6}{5} \frac{x^8}{8!} - \gamma \frac{x^2}{2} \biggr) \biggr) \bigg| \leq \frac{C}{n^{1/4}}.
$$
\end{itemize}

Note that there is some interesting change in limiting behaviour of all of these models
at criticality. While for $\beta<1$ all of the models have the same rate of convergence for the Central Limit
Theorem behaviour, in the limit at
criticality the limiting distribution function as well as the distributions which depend on some
moments of $W$ becomes characteristic of the underlying distribution $\varrho$. Moreover the
rate of convergence differs at criticality (for $k \geq 3$).
\end{example}

\section{Appendix}

\begin{proof}[Proof of Lemma \ref{genbound}]
Consider a probability density of the form
\begin{equation} \label{proto}
p(x) := p_k(x) := b_k \exp \bigl( - a_k x^{2k} \bigr)
\end{equation}
with $b_k = \int_{\R} \exp \bigl(- a_k x^{2k} \bigr) \, dx$. 
Clearly $p$ satisfies Assumption (D). First we prove that the solutions $f_z$ of the Stein equation, which characterizes the distribution with respect
to the density \eqref{proto}, satisfies Assumption (B2). 
%Proving that Assumption (B1) is fulfilled follows similarly and is omitted (see \cite{Stein:1986}
%or \cite{ChenShao:2005} for technical details). Some of the presented bounds are not optimal. We omit technically more involed proofs for
%getting good or even optimal constants. The case where $x^{2k}$ is replaced by a polynomial of degree $2k$ is omitted as well. 
Let $f_z$ be the solution of
$$
f_z'(x) + \psi(x) f_z(x) = 1_{\{x \leq z\}}(x) - P(z).
$$
Here $\psi(x) = - 2k \, a_k \, x^{2k-1}$. We have
\begin{equation} \label{thesolution}
f_z(x) = \left\{ \begin{array}{ll} (1-P(z)) \, P(x) \exp( a_k x^{2k}) b_k^{-1} & \mbox{for } x \leq z, \\
P(z) \, (1-P(x))  \exp( a_k x^{2k}) b_k^{-1} & \mbox{for } x \geq z \\ \end{array} \right.
\end{equation}
with $P(z) := \int_{-\infty}^z p(x) \, dx$. Note that $f_z(x) = f_{-z}(-x)$, so we need only to consider the case $z \geq 0$.
%In the following, all proofs will only be given for $x>0$. The case $x \leq 0$ can be treated analoguously.  
For $x >0$ we obtain
\begin{equation} \label{ungl1}
1 - P(x) \leq \frac{b_k}{2k \, a_k x^{2k-1}} \exp \bigl( - a_k x^{2k} \bigr),
\end{equation}
whereas for $x<0$ we have
\begin{equation} \label{ungl1b}
P(x) \leq \frac{b_k}{2k \, a_k |x|^{2k-1}} \exp \bigl( - a_k x^{2k} \bigr).
\end{equation}
By partial integration we have
$$
\int_x^{\infty} \frac{(2k-1)}{2k \, a_k} t^{-2k} \,  \exp \bigl( - a_k t^{2k} \bigr) = - \frac{1}{2k \, a_k \, t^{2k-1}}  \exp \bigl( - a_k t^{2k} \bigr) 
\bigg|_{x}^{\infty} - \int_{x}^{\infty}  \exp \bigl( - a_k t^{2k} \bigr)\, dt.
$$
Hence for any $x >0$
\begin{equation} \label{ungl2}
b_k \, \biggl( \frac{x}{2k \, a_k x^{2k} + 2k-1} \biggr) \exp \bigl( - a_k x^{2k} \bigr)
\leq  1 - P(x).
\end{equation}
With \eqref{ungl1} we get for $x >0$
$$
\frac{d}{dx} \biggl( \exp \bigl( a_k x^{2k} \bigr) \int_{x}^{\infty} \exp  \bigl( - a_k t^{2k} \bigr) \, dt \biggr) = -1 + 2k \, a_k x^{2k-1} \exp
\bigl(a_k x^{2k} \bigr) \, \int_x^{\infty} \exp  \bigl( - a_k t^{2k} \bigr) \, dt < 0.
$$
So $ \exp \bigl( a_k x^{2k} \bigr) \int_{x}^{\infty} \exp  \bigl( - a_k t^{2k} \bigr) \, dt$ attains its maximum at $x=0$ and therefore
$$
 \exp \bigl( a_k x^{2k} \bigr) \, b_k \, \int_x^{\infty}   \exp  \bigl( - a_k t^{2k} \bigr) \, dt \leq \frac 12.
$$
Summarizing we obtain for $x>0$
\begin{equation} \label{ungl3}
1-P(x) \leq \min \biggl( \frac 12, \frac{b_k}{2k \, a_k \, x^{2k-1}} \biggr)  \exp  \bigl( - a_k x^{2k} \bigr). 
\end{equation}
With \eqref{ungl1b} we get for $x<0$
$$
\frac{d}{dx} \biggl( \exp \bigl( a_k x^{2k} \bigr) \int_{-\infty}^{x} \exp  \bigl( - a_k t^{2k} \bigr) \, dt \biggr) = 1 + 2k \, a_k x^{2k-1} \exp
\bigl(a_k x^{2k} \bigr) \, \int_{-\infty}^x \exp  \bigl( - a_k t^{2k} \bigr) \, dt > 0.
$$
So $ \exp \bigl( a_k x^{2k} \bigr) \int_{-\infty}^{x} \exp  \bigl( - a_k t^{2k} \bigr) \, dt$ attains its maximum at $x=0$ and therefore
$$
 \exp \bigl( a_k x^{2k} \bigr) \, b_k \, \int_{-\infty}^{x}   \exp  \bigl( - a_k t^{2k} \bigr) \, dt \leq \frac 12.
$$
Summarizing we obtain for $x<0$
\begin{equation} \label{ungl3b}
P(x) \leq \min \biggl( \frac 12, \frac{b_k}{2k \, a_k \, |x|^{2k-1}} \biggr)  \exp  \bigl( - a_k x^{2k} \bigr). 
\end{equation}

Applying \eqref{ungl3} and \eqref{ungl3b} gives $0 < f_z(x) \leq \frac{1}{2 \, b_k}$ for all $x$. Note that for
$x<0$ we only have to consider the first case of \eqref{thesolution}, since $z \geq 0$. The constant $\frac{1}{2 \, b_k}$ is not optimal. Following the proof of Lemma 2.2
in \cite{ChenShao:2005} or alternatively of Lemma 2 in \cite[Lecture II]{Stein:1986} would lead to optimal constants. We omit this.
It follows from \eqref{thesolution} that
\begin{equation} \label{thederivative}
f_z'(x) = \left\{ \begin{array}{ll} (1-P(z)) \biggl[ 1 + x^{2k-1} \, 2k \, a_k \,  P(x) \exp( a_k x^{2k}) b_k^{-1} \biggr] & \mbox{for } x \leq z, \\
P(z) \biggl[ (1-P(x))  \, 2k \, a_k \, x^{2k-1} \, \exp( a_k x^{2k}) b_k^{-1} -1 \biggr] & \mbox{for } x \geq z. \\ \end{array} \right.
\end{equation}
With \eqref{ungl1} we obtain for $0< x \leq z$ that
$$
f_z'(x) \leq  (1-P(z)) \biggl[z^{2k-1} \, 2k \, a_k \,  P(z) \exp( a_k z^{2k}) b_k^{-1} \biggr] +1 \leq 2.
$$
The same argument for $x \geq z$ leads to $|f_z'(x)| \leq 2$. For $x<0$ we use the first half of \eqref{thesolution} and apply \eqref{ungl1b} 
to obtain $|f_z'(x)| \leq 2$. Actually this bound will be improved later.
Next we calculate the derivative of $-\psi(x) \, f_z(x)$:
\begin{equation} \label{theproduct}
(-\psi(x) f_z(x))'
= \left\{ \begin{array}{ll} \frac{(1-P(z))}{b_k} \biggl[ P(x)e^{a_k x^{2k}} \biggl(2k(2k-1)a_k x^{2k-2}+(2k)^2 a_k^2 x^{4k-2} \biggr) + 2k a_k x^{2k-1} b_k \biggr],
&  x \leq z, \\
\frac{P(z)}{b_k} \biggl[ (1-P(x)) e^{a_k x^{2k}} \biggl(2k(2k-1) a_k x^{2k-2} +(2k)^2 a_k^2 x^{4k-2} \biggr) - 2k a_k x^{2k-1} b_k \biggr],
& x \geq z. \\ \end{array} \right.
\end{equation}
With \eqref{ungl2} we obtain $(-\psi(x) f_z(x))' \geq 0$, so $-\psi(x) f_z(x)$ is an increasing function of $x$ (remark that for $x<0$
we only have to consider the first half of \eqref{thesolution}). Moreover
with \eqref{ungl1}, \eqref{ungl1b} and \eqref{ungl2} we obtain that
\begin{equation} \label{ungl4}
\lim_{x \to -\infty} 2k \, a_k \, x^{2k-1} f_z(x) = P(z)-1 \quad \text{and} \quad \lim_{x \to \infty}  2k \, a_k \, x^{2k-1} f_z(x) = P(z).
\end{equation}
Hence we have $|2k \, a_k \, x^{2k-1} f_z(x)| \leq 1$ and $|2k \, a_k \bigl(x^{2k-1} f_z(x)- u^{2k-1} f_z(u)\bigr)| \leq 1$ for any $x$ and $u$.
From \eqref{ungl1} it follows that $f_z'(x) >0$ for all $x<z$ and $f_z'(x) < 0$ for $x>z$.
With Stein's identity $f_z'(x) = - \psi(x) f_z(x) + 1_{\{x \leq x\}} -P(z)$ and
\eqref{ungl4} we have
$$
0 < f_z'(x) \leq -\psi(z) f_z(z) + 1 - P(z) <1 \quad \text{for} \quad x < z
$$
and
$$
-1 < -\psi(z) f_z(z) - P(z) \leq f_z'(x) < 0  \quad \text{for} \quad x > z.
$$
Hence, for any $x$ and $y$, we obtain
$$
|f_z'(x)| \leq 1 \quad \text{and} \quad |f_z'(x) - f_z'(y)| \leq \max \bigl( 1, -\psi(z) f_z(z) + 1 - P(z) - (-\psi(z) f_z(z) - P(z)) \bigr)=1.
$$
Next we bound $(-\psi(x) f_z(x))'$. We already know that $(-\psi(x) f_z(x))'>0$. Again we apply \eqref{ungl1} and \eqref{ungl1b} to see that
$$
(-\psi(x) f_z(x))' \leq \frac{2k-1}{|x|}
$$
for $x \geq z >0$ and all $x \leq 0$. For $0<x \leq z$ this latter bound holds, as can be seen by applying this bound
(more precisely the bound for $(-\psi(x) f_z(x))' \, \frac{b_k}{P(z)}$ for $x \geq z$) with $-x$ for $x$ to the formula for $(\psi(x) f_z(x))'$ 
in $x \leq z$. For some constant $c$ we can bound $(\psi(x) f_z(x))'$ by $c$ for all $|x| \geq \frac{2k-1}{c}$. 
Moreover, on $[-\frac{2k-1}{c}, \frac{2k-1}{c}]$ the continuous
function  $(-\psi(x) f_z(x))'$ is bounded by some constant $d$, hence we have proved
$$
|-(\psi(x) f_z(x))' | \leq \max(c,d).
$$
The problem of finding the optimal constant, depending on $k$, is omitted. Summarizing, Assumption (B2) is fulfilled for $p$ with
$d_2=d_3=1$ and some constants $d_1$ and $d_4$.
\medskip

Next we consider an absolutely continuous function $h : \R \to \R$. Let $f_h$ be the solution of the Stein equation \eqref{steinid2},
that is
$$
f_h(x) = \frac{1}{p(x)} \int_{-\infty}^x (h(t) - Ph) \, p(t) \, dt = - \frac{1}{p(x)} \int_{x}^{\infty} (h(t) - Ph) \, p(t) \, dt. 
$$
We adapt the proof of \cite[Lemma 2.3]{ChenShao:2005}: without loss of generality we assume that $h(0)=0$ and put $e_0 := \sup_{x} |h(x) -Ph|$ and
$e_1:= \sup_x |h'(x)|$. Form the definition of $f_h$ it follows that $|f_h(x)| \leq e_0 \frac{1}{2 b_k}$. 
An alternative bound is $c_1 \, e_1$ with some constant $c_1$ depending on $\E|Z|$, where $Z$ denotes a random variable distributed according to $p$.
With \eqref{steinid2} and \eqref{ungl2}, for $x \geq 0$,
$$
|f_h'(x)| \leq |h(x) -Ph| - \psi(x) e^{a_k x^{2k}} \int_x^{\infty} |h(t) - Ph| e^{-a_k t^{2k}} \, dt \leq 2 e_0.
$$
An alternative bound is $c_2 \, e_1$ with some constant $c_2$ depending on the $(2k-2)$'th moment of $p$. This is using Stein's identity
\eqref{steinid2} to obtain
$$
f_h'(x) = - e^{a_k x^{2k}} \int_{x}^{\infty} (h'(t) - \psi'(t) \, f(t)) e^{-a_k t^{2k}} \, dt.
$$
The details are omit. To bound the second derivative $f_h''$, we differentiate \eqref{steinid2} and have
$$
f_h''(x) = \bigl(\psi^2(x) - \psi'(x) \bigr) f_h(x) - \psi(x) \bigl( h(x) -Ph \bigr) + h'(x).
$$
Similarly to \cite[(8.8), (8.9)]{ChenShao:2005} we obtain
$$
h(x) - Ph = \int_{-\infty}^x h'(t) P(t) \, dt - \int_{x}^{\infty} h'(t) (1-P(t)) \, dt.
$$
It follows that
$$
f_h(x) = -\frac{1}{b_k} e^{a_k x^{2k}} (1-P(x)) \, \int_{-\infty}^x h'(t) P(t) \, dt - \frac{1}{b_k} e^{a_k x^{2k}} P(x) \, \int_{x}^{\infty} h'(t) (1-P(t))\, dt.
$$
Now we apply the fact that the quantity in \eqref{theproduct} is non-negative to obtain 
\begin{eqnarray*}
|f_h''(x)| & \leq & |h'(x)| +  \big| \bigl(\psi^2(x) - \psi'(x) \bigr) f_h(x) - \psi(x) \bigl( h(x) -Ph \bigr) \big| \\
& \leq & |h'(x)| + \biggl| \biggl( - \psi(x) -  \frac{1}{b_k} \bigl( \psi^2(x) - \psi'(x) \bigr) e^{a_k x^{2k}} (1-P(x)) \biggr) \, \int_{-\infty}^x h'(t) P(t) \, dt \biggr| \\
& & \hspace{0.5cm} +   \biggl| \biggl( \psi(x) -  \frac{1}{b_k} \bigl( \psi^2(x) - \psi'(x) \bigr) e^{a_k x^{2k}} \, P(x) \biggr) \, 
\int_x^{\infty} h'(t) (1-P(t)) \, dt \biggr| \\
& \leq & |h'(x)| + e_1 \biggl( \psi(x) +  \frac{1}{b_k} \bigl( \psi^2(x) - \psi'(x) \bigr) e^{a_k x^{2k}} (1-P(x)) \biggr) \, \int_{-\infty}^x P(t) \, dt \\
& & \hspace{0.5cm} + e_1 \biggl( - \psi(x) +  \frac{1}{b_k} \bigl( \psi^2(x) - \psi'(x) \bigr) e^{a_k x^{2k}} P(x) \biggr) \, \int_{x}^{\infty} (1-P(t)) \, dt.
\end{eqnarray*}
Moreover we know, that the quantity in \eqref{theproduct} can be bounded by $\frac{2k-1}{|x|}$, hence
$$
|f_h''(x)| \leq e_1 + e_1 \frac{2 b_k \, (2k-1)}{|x|} \biggl( \int_{-\infty}^x P(t) \, dt +  \int_{x}^{\infty} (1-P(t)) \, dt \biggr).
$$
Now we bound
$$
\bigl|  \int_{-\infty}^x P(t) \, dt +  \int_{x}^{\infty} (1-P(t)) \, dt \bigr| = \bigl| xP(x) - x(1-P(x)) + 2 \int_x^{\infty} t p(t) \, dt \bigr| \leq 2|x| + 2 \E|Z|,
$$
where $Z$ is distributed according to $p$. Summarizing we have $|f_h''(x)| \leq c_3 \sup_{x} |h'(x)|$ for some constant $c_3$, using
the fact that $f_h$ and therefore $f_h'$ and $f_h''$ are continuous. Hence $f_h$ satisfies Assumption (B1).
\end{proof}

\begin{proof}[Sketch of the proof of Remark \ref{gibbsrem}]
Now let $p(x) = b_k \exp \bigl( - a_k V(x) \bigr)$ and $V$ satisfies the assumptions listed in Remark \ref{gibbsrem}. To proof
that $f_z$ (with respect to $p$) satisfies Assumption (B2), we adapt \eqref{ungl2} as well as
\eqref{ungl3} and \eqref{ungl3b}, using the assumptions on $V$. We obtain for $x >0$
\begin{equation*} 
b_k \, \biggl( \frac{V'(x)}{V''(x) + a_k V'(x)^2} \biggr) \exp \bigl( - a_k \, V(x) \bigr)
\leq  1 - P(x).
\end{equation*}
and for $x >0$
\begin{equation*} 
1-P(x) \leq \min \biggl( \frac 12, \frac{b_k}{a_k \, V'(x)} \biggr)  \exp  \bigl( - a_k \, V(x) \bigr)
\end{equation*}
and for $x < 0$
\begin{equation*} 
P(x) \leq \min \biggl( \frac 12, \frac{b_k}{a_k \, |V'(x)|} \biggr)  \exp  \bigl( - a_k \, V(x) \bigr). 
\end{equation*}
Estimating $(-\psi(x) f_z(x))'$ gives
$$
(-\psi(x) f_z(x))' \leq {\rm const.} \, \frac{V''(x)}{|V'(x)|}.
$$
By our assumptions on $V$, the right hand side can be bounded for $x \geq d$ with $d \in \R_+$ and since
$\psi(x) f_z(x)$ is continuous, it is bounded everywhere. 
\end{proof}

{\em Acknowledgement.}  
During the preparation of our manusscript we became aware of a preprint of S. Chatterjee ans Q.-M. Shao about
Stein's method with applications to the Curie-Weiss model.  As far as we understand, there the authors
give an alternative proof of Theorem 1.2 and 1.3.

\bibliographystyle {amsplain}
\bibliography{08.10.2007}

%\end{thebibliography}

\end{document}